\documentclass[11pt]{article}

\usepackage[utf8]{inputenc}

\usepackage{amsthm}
\usepackage{tensor}
\usepackage{mathtools}
\usepackage{amsfonts}
\usepackage{amsmath}
\usepackage{amssymb}
\usepackage{upgreek}
\usepackage{multirow}
\usepackage{mathtools}
\usepackage{enumerate}
\usepackage{listings}
\usepackage{authblk}

\usepackage{tikz}
\usepackage{tikz-cd}
\usetikzlibrary{matrix,arrows}
\newtheorem{theorem}{Theorem}
\newtheorem{properties}{Property}

\newtheorem{proposition}{Proposition}

\newtheorem{conjecture}{Conjecture}

\newtheorem{corollary}{Corollary}
\newtheorem{claim}{Claim}

\newtheorem{remark}{Remark}

\DeclareMathOperator{\Hom}{Hom}
\DeclareMathOperator{\Span}{span}
\DeclareMathOperator{\Dim}{dim}
\def\bo#1{\textrm{bo}_c(#1)}
\def\bomin#1{\textrm{bo}_{\textrm{min}}(#1)}
\def\bomax#1{\textrm{bo}_{\textrm{max}}(#1)}
\def\bw#1{\textrm{bw}(#1)}
\def\B#1#2{\mathcal{B}_{#1}^{#2}}

\def\cm{\varphi}
\def\mm{m}
\def\mmc{\overline{m}}
\def\aa{\boldsymbol{\alpha}}
\def\be{\begin{equation}}
\def\ee{\end{equation}}

\begin{document}
	
%\author{{\large Sigbj\o rn Hervik$^\diamond$, Marcello Ortaggio$^\star$ and Lode Wylleman$^{\diamond,\dagger,\natural}$}\\
%	\vspace{0.05cm} \\
%	{\small $^\diamond$ Faculty of Science and Technology, University of Stavanger}, {\small  N-4036 Stavanger, Norway}  \\
%	{\small $^\star$ Institute of Mathematics, Academy of Sciences of the Czech Republic}, {\small \v Zitn\' a 25, 115 67 Prague 1, Czech Republic} \\
%	{\small $^\dagger$ Faculty of Applied Sciences TW16, Ghent
%		University},  {\small  Galglaan 2, 9000 Gent, Belgium}\\
%	{\small $^\natural$ Department of Mathematics, Utrecht University,
%		Budapestlaan 6, 3584 CD Utrecht, The Netherlands}\\
%	{\small E-mail: \texttt{sigbjorn.hervik@uis.no,
%			ortaggio@math.cas.cz, lode.wylleman@ugent.be}} }	
	
\author{Matthew Terje Aadne$^\dagger$ and Lode Wylleman$^{\dagger,\diamond}$}
%\title{Proof of the Kundt Theorem in 4D}
\title{Progress on the Kundt conjecture}
\affil{${}^\dagger$Faculty of Science and Technology,\\
University of Stavanger,\\
4036 Stavanger, Norway\\
${}^\diamond$ Faculty of Engineering Sciences and Architecture,\\
Ghent University,\\
9052 Ghent, Belgium
}
\affil{matthew.t.aadne@uis.no,lode.wylleman@ugent.be}
\maketitle
\begin{abstract}
    The Kundt conjecture states that a Lorentzian manifold of arbitrary dimension which is not characterized by its scalar polynomial curvature invariants (SPIs) allows for a non-twisting, non-shearing and non-expanding (in short, Kundt) null congruence of geodesics. The conjecture has been proven for dimensions 3 and 4. A necessary condition for a spacetime not to be characterized by SPIs is that all covariant derivatives of the Riemann tensor are of aligned type II or more special in the null alignment classification. In arbitrary dimensions, we prove that this property indeed requires the presence of a Kundt null congruence when a certain genericity condition holds, or when the trac-free Ricci or Weyl tensor is of genuine type III or N, thus confirming the validity of the Kundt conjecture in these cases. We also strenghten the results for dimensions 3 and 4 by removing regularity assumptions and showing that only the third covariant derivative is needed to obtain the results. A key tool in our proofs is a new bilinear map acting on tensors of related boost orders relative to a null direction. 
    %Here we provide streamlined proofs of a reformulation of the conjecture in terms of null alignment of curvature tensors; and strengthen the results by removing regularity assumptions.  
    %We strengthen the Kundt theorem in dimensions three and four by removing regularity assumptions. 
    %We also prove the conjecture in arbitrary dimensions for the subcases where the Ricci or Weyl tensor is of null alignment type III or N, and confirm its validity under a certain genericity condition.  
\end{abstract}
\section{Introduction}

We consider Lorentzian manifolds (or spacetimes) of arbitrary dimension $n$. To describe the topic of this paper let us first recall some concepts and nomenclature. 
%The results in this paper involve null alignment theory of tensors; the reader unfamiliar with this topic is referred to the review \cite{OrtPraPrareview} or the appendix. Following \cite{Hervik-align} we will say that two or more tensors are {\em aligned type II or more special} w.r.t.~a null congruence $c$ if they have non-positive boost order along $c$. 
A tensor (or set of tensors) is said to be of {\em aligned type II or more special}~\cite{Milsonetal05,Hervik-align} 
or {\em algebraically special}~\cite{OrtPraPrareview} 
w.r.t.~a null congruence $c$ if it has (or each member has) non-positive boost order along $c$ (see \cite{OrtPraPrareview} or appendix \ref{app:nullalignment} for a review on boost order and null alignment types of general tensors). The Riemann tensor and its covariant derivatives will be jointly called {\em curvature tensors}. 
%(This definition involves null alignment theory of tensors~\cite{Milsonetal}; readers unfamiliar with this theory are referred to the review \cite{OrtPraPrareview} or the appendix.) Tensors which can be obtained by (anti)symmetrization and taking tensor products, real linear combinations and contractions from the metric, the inverse metric and the Riemann tensor and its consecutive covariant derivatives, will be jointly dubbed {\em curvature tensors}. 
A null congruence $c$ on a spacetime will be dubbed {\em Kundt} if it is non-twisting (thus geodetic), non-shearing and non-expanding; if $k$ is any vector field that generates $c$ the Kundt condition may be succinctly written as 
$$
k_{[a}\nabla_{b]}k_{[c}k_{d]}=0.
$$ 
A spacetime is called 
%Kundt if it allows for a Kundt null congruence $c$, and {\em degenerate Kundt} if all curvature tensors are moreover aligned type II or more special w.r.t.~$c$. 
{\em degenerate Kundt}~\cite{ColeyetalKundt} if it allows for a Kundt null congruence w.r.t.~which all curvature tensors are of aligned type II or more special. %w.r.t.~some Kundt null congruence. 

The following conjecture, stated here in the terminology of this article, was put forward in \cite{SCPI}:
\begin{conjecture}\label{Conj:Kundt}
If all curvature tensors of a spacetime $(M,g)$ are of aligned type II or more special 
w.r.t.~some null congruence then the spacetime is degenerate Kundt, i.e., for each point $p\in M$ the curvature tensors are of aligned type II or more special w.r.t.~a {\em Kundt} null congruence $c^{\prime}$ defined in a neighborhood of $p$. % there exists a null congruence $c^{\prime}$ defined in a neighborhood of $p$ that is Kundt.
%the curvature tensors are of aligned type II or more special 
%w.r.t.~some
%w.r.t.~a (possibly different) 
%{\em Kundt} null congruence. then about each point $p\in M,$ there exists a null-line distribution $c^{\prime}$ defined in a neighborhood of $p$ that is Kundt.
\end{conjecture}

%As a first motivation, t
A main motivation to prove this conjecture lies in 
%The confirmation of this conjecture is important with regard to 
the invariant characterization of spacetime metrics and the equivalence problem~\cite{Olverbook,Stephanibook}. Given an open subset $U$ of a manifold, two metrics with components $g_{\alpha\beta}$ and $g'_{\mu\nu}$ in respective coordinate systems $\{x^\alpha\}$ and $\{x'^\mu\}$ on $U$ are locally equivalent if and only if a coordinate transformation, in shorthand $x' = x'(x)$, 
%$x'^\mu=x'^\mu(x^\alpha)$ [in shorthand, $y = y(x)$] 
exists such that $g_{\alpha\beta}(x)=\tilde{g}_{\mu\nu}(x'(x))\frac{\partial x'^\mu}{\partial x^\alpha}(x)\frac{\partial x'^\nu}{\partial x^\beta}(x)$. To verify in a direct way whether such a coordinate transformation exists or not may be extremely difficult. A more feasible method is to look for a set of {\em scalar invariants} associated to any metric, having the property that if $f_A(x)$ and $f'_A(x')$ are corresponding invariants in the set respectively calculated for $g(x)$ and $g'(x')$, then these metrics are equivalent precisely when $f_A(x)=f'_A(x'(x))$ for all $A$. Such a set thus gives a complete local characterization of any spacetime. It is well known that a set of {\em Cartan invariants} with this property can be constructed from a finite number of curvature tensors~\cite{Stephanibook}. 
%the Riemann tensor and a finite number of its covariant derivatives. 
This construction, however, is sometimes tedious (and for dimensions larger than four even not known in general). It is natural to ask whether the simplest thinkable invariants, namely {\em scalar polynomial invariants} (SPIs) constructed from the curvature tensors 
%Riemann tensor and its covariant derivatives (jointly called {\em curvature tensors} henceforth) 
by index shuffling, taking tensor products 
%(anti)symmetrizations, 
and full scalar contractions, completely characterize all spacetime metrics as well. However, the answer is negative. Elementary counterexamples are provided by metrics for which all SPIs vanish. Unlike their Riemannian counterparts, such Lorentzian metrics are generally not locally equivalent to the Minkowski metric; the corresponding spacetimes have been dubbed {\em VSI spacetimes} and include the so-called {\em pp-waves}~\cite{Brinkmann25,Ortaggio18}. The VSI theorem~\cite{Pravdaetal04,Coleyetal04,Hervik-align,OrtPraPrareview} now states that 
%the non-trivial VSI spacetimes form an ample class, with metrics not locally equivalent to one another nor to the Minkowski metric. 
%The algebraic VSI theorem implies that a spacetime is VSI precisely when all curvature tensors are of type III or more special in the alignment classification 
%First in dimension 4 [REFS] and then in arbitrary dimension [REFS] the {\em VSI theorem} was proven: 
a spacetime is VSI if, and only if, %necessarily possess 
it allows for a Kundt null congruence w.r.t.~which 
%along which the boost order of the Riemann tensor (and then of any curvature tensor) is negative. 
the Riemann tensor (and then any curvature tensor) is type III or more special 
in the null alignment classification. On itself, the assumption that all curvature tensors are of aligned type III or more special trivially implies the VSI property, such that the `only if' part of the VSI theorem then confirms conjecture \ref{Conj:Kundt} for this case. 
%; furthermore, this property is also sufficient for the spacetime to be VSI. 
VSI spacetimes are special instances of degenerate Kundt spacetimes. More generally, it was shown in \cite{SCPI} that all of these provide examples of spacetimes that are  
%degenerate Kundt spacetimes provide  
not locally characterized by their SPIs. 
%One may wonder whether non-VSI spacetimes exist for which the SPIs do not provide a full local characterization. In \cite{SCPI} it was shown that the degenerate Kundt spacetimes form a class of examples. 
The {\em Kundt conjecture} now states that they are the only examples: a spacetime not locally characterized by SPIs must be degenerate Kundt. In \cite{CharBySpi} this conjecture was proven in dimension 4, while arguments (but not clean proofs) were given in \cite{SCPI} supporting the conjecture in arbitrary dimension.
%Note that the `only if' part of the above mentioned VSI result confirms the conjecture in the case where all curvature tensors are of aligned type III or more special in the null alignment classification, as this assumption implies the VSI property. 
%Most remarkably, the authors of \cite{CharBySpi} proved that in dimension 4 the converse is also true: if a spacetime is not locally characterized by SPIs then it must be degenerate Kundt. Arguments (but not clean proofs) were given in \cite{SCPI} supporting the conjecture (which can be dubbed the {\em Kundt conjecture}) for arbitrary dimension. that this result remains true in dimensions higher than 4. 
From Hervik's alignment theorem~\cite{Hervik-align} it follows that a spacetime not fully characterized by its SPIs must have curvature tensors of aligned type II or more special (see also \cite{HerOrtWyl13} for a streamlined proof and further explanations). Hence, the validity of conjecture \ref{Conj:Kundt} would imply that of the Kundt conjecture. 
%show that {\em any} spacetime not characterized by SPIs is a degenerate Kundt spacetime or, put differently, 
This would mean that SPIs {\em do} locally characterize spacetimes completely, {\em except} the degenerate Kundt spacetimes. Since SPIs are much more easy to calculate than Cartan invariants, this would considerably simplify the local characterization of spacetimes. 

The present paper confirms conjecture \ref{Conj:Kundt}, and thus proves the Kundt conjecture, in important subcases. The structure is as follows. 
In section \ref{sec:bilinearmap} we introduce a bilinear map acting on tensors of related boost orders relative to a null direction. This map underlies a particular factorization result for covariant derivatives of tensors (theorem \ref{FactorizationTheorem}) and forms a key tool in our proofs. It is used to show the validity of conjecture \ref{Conj:Kundt} under a genericity condition on the Ricci tensor, specified in section \ref{sec:Kundtgeneric}. %genericity condition holds. 
As in dimension 4~\cite{CharBySpi} one can split up the verification of the conjecture for arbitrary dimension according to the (genuine) null alignment types of the trace-free Ricci and Weyl tensors; in section \ref{sec:IIIorN} we provide proofs for the cases where only one of these tensors is assumed to be of type III or N, thus generalizing the above VSI result in this respect. In sections \ref{sec:Kundt3} and \ref{sec:Kundt4} we turn our attention to dimensions 3 and 4, and strenghten the proofs given in the literature by removing usual regularity assumptions on the null alignment type, and showing that only the third covariant derivative of the Riemann tensor is needed to obtain the results (opposed to the fourth derivative exployed for some subcases in \cite{CharBySpi}). %We end with a discussion and future prospect in section \ref{sec:discussion}. 
Appendix \ref{app:nullalignment} provides a succinct review of boost order and null alignment type theory of tensors.\\ 

{\em General notation.} For tensors we use index-free or abstract index notation, depending on the context. As usual, abstract indices are lowered and raised by the metric $g_{ab}$ resp.~the inverse metric $g^{ab}$ ($g_{ab}g^{bc}=\delta_a^c$) in use, leading to geometrically equivalent tensors denoted by the same symbol. Einstein's summation convention is applied on both abstract and frame indices.

\section{Algebraic type based bilinear maps}\label{sec:bilinearmap}
In this section we define a bilinear map between spaces of rank $r$ tensors of specified boost orders. We show that this map, when applied to a Lorentzian manifold with a given null-congruence dictates the geometry of the null congruence.

Let $(V,g)$ be a Lorentzian space, consisting of a real vector space $V$ with inner product $g$, and let ${\cal T}_r$ denote the space of covariant rank $r$ tensors ($r\geq 1$) and $\Lambda^2$ the space of 2-forms over $V$
%\footnote{For tensors over $V$ and the dual space $V^*$ we use index-free or abstract index notation, depending on the context. As usual, abstract indices are lowered and raised by the metric $g\equiv g_{ab}$ resp.~the inverse metric $g^{-1}\equiv g^{ab}$ ($g_{ab}g^{bc}=\delta_a^c$) in use, leading to geometrically equivalent tensors denoted by the same symbol. Einstein's summation convention is applied on both abstract and frame indices}. %Round \{square\} brackets surrounding indices indicate \{anti\}symmetrization.
%Let ${\cal T}_r$ denote the space of covariant rank $r$ tensors ($r\geq 1$) and $\Lambda^2$ the space of 2-forms over $V$. 
%, and put $\DD\equiv \bigcup_{r\geq 1} {\cal T}_r\times {\cal T}_r$. 
For all $X,Y\in{\cal T}_r$ we define $\cm(X,Y)\in \Lambda^2$ by
%$
%\cm:\bigoplus_{r\geq 1} {\cal T}_r\times {\cal T}_r\rightarrow \Lambda^2
%$
%by 
\begin{equation*}
\cm(X,Y)_{ab}=\sum_{i=1}^r \tensor{X}{_{c_{1}}_{\cdots} _{c_{i-1}}_b_{\;c_{i+1}}_{\cdots}_{c_{r}}}\tensor{Y}{^{c_1}^{\cdots}^{c_{i-1}}_{\;a}^{c_{i+1}}^{\cdots}^{c_{r}}}-\tensor{X}{_{c_{1}}_{\cdots} _{c_{i-1}}_a_{\;c_{i+1}}_{\cdots}_{c_{r}}}\tensor{Y}{^{c_1}^{\cdots}^{c_{i-1}}_{\;b}^{c_{i+1}}^{\cdots}^{c_{r}}}.
\end{equation*}
%for all $X,Y\in{\cal T}_r$. 

Consider now a triple $(V,g,c)$ where $c$ is a null line in $(V,g)$, and let $\B{r}{s}$ be the subspace of ${\cal T}_r$ consisting of the  rank $r$ tensors of boost order $\leq s$ along $c$.  Identifying vectors with their gometrically equivalent covectors we have $\B{1}{0}=c^\perp$ and $\B{1}{-1}=c$. Note that if $\bo{T}=s$ and $\bo{Q}=s'$ then $\cm(T,Q)\in \B{2}{s+s'}$; applying this to $s'=-s-1$ gives $\cm(T,Q)\in \B{2}{-1}$, implying $\cm(T,Q)^{ab}v_{b}\in c^\perp$ for any $v\in V$ and $\cm(T,Q)^{ab}Y_b\in c$ for any $Y\in c^\perp$, while if $s'<-s-1$ we have $\cm(T,Q)\in\B{2}{-2} \cap\Lambda^2$ and thus $\cm(T,Q)=0$. 
%(since the minimum boost order of a non-zero 2-form is -1). 
%(since $\bo{A}\geq -1$ for any $0\neq A\in\Lambda^2$).
Fixing a non-zero element $k\in c$, let $l$ be any null vector satisfying $g(k,l)=1$ and define, for each $r\geq 1$ and $s$ with $-r\leq s\leq r-1$, the bilinear map
\begin{equation}\label{def-bilinear-a}
\langle\cdot\vert k \vert\cdot\rangle:\mathcal{B}^{s}_{r}\times \mathcal{B}^{-s-1}_{r} \rightarrow c^{\perp}/c
\end{equation}
%which transforms $(T,Q)\in \mathcal{B}^{s}_{r}\times \mathcal{B}^{-s-1}_{r}$ to
by
\begin{equation}\label{def-bilinear-b}
\langle T\vert k\vert Q\rangle^a=\pi(\cm(T,Q)^{ab}l_b),
\end{equation}
%Fixing a non-zero element $k\in c$, let $l$ be any null vector satisfying $g(k,l)=1$ and define the bilinear map $\langle\cdot\vert k \vert\cdot\rangle$ transforming, for each $r$ and $s$, the element $(T,Q)\in \B{r}{s}\times \B{r}{-s-1}$ to $\langle T\vert k\vert Q\rangle\in c^\perp/c$ given by 
%\begin{equation}\label{def-bilinear-a}
%    \langle\cdot\vert k \vert\cdot\rangle:\bigcup_{r\geq 1}\bigcup_{s=-r}^{r-1}\B{r}{s}\times \B{r}{-s-1} \rightarrow c^{\perp}/c
%\end{equation}
%\begin{equation}\label{def-bilinear}
%\begin{gathered}
%    \langle T\vert k\vert Q\rangle= \\ \pi(\sum_{i=1}^{r}[-\tensor{T}{_{c_{1}}_{\cdots} _{c_{i-1}}^b_{\;c_{i+1}}_{\cdots}_{c_{r}}}\tensor{Q}{^{c_1}^{\cdots}^{c_{i-1}}_{d}^{c_{i+1}}^{\cdots}^{c_{r}}}+\tensor{T}{_{c_{1}}_{\cdots} _{c_{i-1}} _d_{c_{i+1}} _{\cdots}_ {c_{r}}}\tensor{Q}{^{c_1}^{\cdots}^{c_{i-1}}^{b}^{c_{i+1}}^{\cdots}^{c_{r}}}]l^{d}),
%    \end{gathered}
%\end{equation}
for all $T\in \B{r}{s}$ and $Q\in \B{r}{-s-1}$, 
where $\pi:c^{\perp}\rightarrow c^{\perp}/c$ is the quotient map. (Note that, with a slight abuse of notation, we also employ $a,b,\cdots$ as abstract indices for tensors over $c^\perp/c$.) This bilinear map is independent of the choice of $l$: if $\hat{l}$ is another null-vector satisfying $g(k,\hat{l})=1$ there exists $Y\in c^{\perp}$ such that $\hat{l}=l+Y$, and since $\cm(T,Q)^{ab}Y_b\in c$ we have $\pi(\cm(T,Q)^{ab}l_b)=\pi(\cm(T,Q)^{ab}\hat{l}_b)$. 
%an inspection of the algebraic type shows that 
%the contractions
%\begin{equation*}
%    -\tensor{T}{_{c_{1}}_{\cdots} _{c_{i-1}}^b_{\;c_{i+1}}_{\cdots}_{c_{r}}}\tensor{Q}{^{c_1}^{\cdots}^{c_{i-1}}_{d}^{c_{i+1}}^{\cdots}^{c_{r}}}Y^{d}\quad\textrm{and}\quad
%%\end{equation}
%%\begin{equation}
%    \tensor{T}{_{c_{1}}_{\cdots} _{c_{i-1}} _d_{\;c_{i+1}} _{\cdots}_ {c_{r}}}\tensor{Q}{^{c_1}^{\cdots}^{c_{i-1}}^{b}^{c_{i+1}}^{\cdots}^{c_{r}}}Y^{d}
%\end{equation*}
%have boost order $\leq s+(-s-1)=-1$ along $c$, i.e., they are aligned with $c$ and therefore annihilated by the quotient map.
In fact, we could also drop the dependence on $k\in c$ by considering the class of maps $\langle \cdot\vert k \vert \cdot \rangle$ in the projective space
\begin{equation}
\mathbb{P}(\Hom(\mathcal{B}_{r}^{s}\otimes \mathcal{B}_{r}^{-s-1},c^{\perp}/c)).
\end{equation}

For given $T\in {\cal T}_r$ with $\bo{T}=s$ the partial map $\langle T|k|\cdot\rangle:\B{r}{-s-1}\rightarrow c^\perp/c$ will be frequently applied in this paper.  
By the above we have
%Note that $\cm(T,Q)_{ab}$ has boost order $\leq s+(-s-1)=-1$ along $c$. This implies that $\cm(T,Q)_{ab}l^b$ has boost order $\leq 0$ along $c$, i.e., $\cm(T,Q)^{ab}l_b$ belongs to $c^\perp$ indeed; if $\bo{Q}< -s-1$ then $\cm(T,Q)_{ab}$ has boost order $< -1$ along $c$ and thus vanishes (since the minimum boost order of a non-zero 2-form is -1), 
%i.e., $\cm(T,Q)^{ab}l_b$ is aligned with $c$, 
%therefore $\langle T\vert k\vert Q\rangle=0$: for given $T$ with $\bo{T}=s$ one needs $\bo{Q}=-s-1$ in order for $\langle T\vert k\vert Q\rangle$ to be non-zero. 
\begin{properties}\label{prop-zero} If $\bo{T}=s$ and $\bo{Q}<-s-1$ then $\langle T|k|Q\rangle=0$. %, i.e., $\bo{Q}=-s-1$ is needed for $\langle T\vert k\vert Q\rangle$ to be non-zero. 
%If $T$ has boost order $s$ along $c$ and $Q$ has boost order $<-s-1$ along $c$ then $\langle T|k|Q\rangle=0$.
\end{properties}
\noindent We will use suitable frame representations of $\langle T|k|\cdot\rangle$. % relative to suitable bases.  
Let $\{e_\alpha\}\equiv\{e_0,e_1,e_j\}=\{k,l,m_j\}$ 
%$\{e_{0},\dots e_{n-1} \}=\{k,l,m_2,\dots m_{n-1}\}$
%Let $\{e_\alpha^a\}=\{e_0^a,e_1^a,e_i^a)=(k^a,l^a,m_i^a)$ 
be a (real) null frame of $(V,g_{ab})$ completing $k$, where $g(k,l)=g(m_j,m_j)=1$ are the only non-zero inner products among the basis vectors.
%, with dual frame $(l_a,k_a,m^i_a=(m_i)_a)$; 
%here and below frame labels $\alpha,\,\beta,\ldots$ run from 0 to $n-1$ while $i,j,\cdots$ run from 2 to $n-1$. 
Here and below (indexed) Greek frame labels run from 0 to $n-1$, while the spatial frame labels $j,\,k,\,\ldots$ run from $2$ to $n-1$. For convenience we will use the following shorthand notation for multi-indices:
$$
\aa\equiv \alpha_1\cdots\alpha_r,\quad \aa_i\gamma\equiv \alpha_1\cdots\alpha_{i-1}\gamma\alpha_{i+1}\cdots \alpha_r,\quad \aa_i{}^\gamma\equiv \alpha_1\cdots\alpha_{i-1}{}^\gamma\alpha_{i+1}\cdots \alpha_r.
$$ 
The induced bases of $\B{r}{-s-1}$ and $c^\perp/c$ are, respectively, $\{e_{\alpha_1}\cdots e_{\alpha_r}|\bw{\aa}\geq s+1\}$ and $\{\pi(m_j)\}$.
%; here and below (indexed) Greek frame labels run from 0 to $n-1$ while $j,\,k,\,\ldots$ run from $2$ to $n-1$. 
Relative to these bases we have $\langle T|k|e_{\alpha_1}\cdots e_{\alpha_r}\rangle=0$ if $\bw{\aa}> s+1$ by property \ref{prop-zero}, while for $\bw{\aa}=s+1$ we obtain
%We will apply the map \eqref{def-bilinear-a}-\eqref{def-bilinear-b} for fixed $T\in\B{r}{s}$ and varying $Q\in\B{r}{-s-1}$. 
%Since $\cm(T,Q)_{ab}\in\Lambda^2$ the 
%partial map that takes $Q\in\B{r}{-s-1}$ to 
%vector $\cm(T,Q)^{ab}l_b$ is orthogonal to $k$ and $l$ for any $Q$, and we obtain the frame expansion  
%\begin{equation}
%\cm(T,e_{\alpha_1}\cdots e_{\alpha_r})_a=\sum_{i=1}^r\left(T_{\alpha_1\cdots \alpha_{i-1} 1 \alpha_{i+1}\cdots \alpha_{r}}\delta_{\alpha_i j}-T_{\alpha_1\cdots \alpha_{i-1} j \alpha_{i+1}\cdots \alpha_{r}}\delta_{\alpha_i0}\right)m^j_a\;.
%\end{equation}
%\begin{equation}
%\cm(T,e_{\alpha_1}\cdots e_{\alpha_r})^a=\sum_{i=1}^r\left(T_{\alpha_1\cdots \alpha_{i-1} 1 \alpha_{i+1}\cdots \alpha_{r}}m_{\alpha_i}^a -\delta_{\alpha_i 0} T_{\alpha_1\cdots \alpha_{i-1}}{}^j{}_{\alpha_{i+1}\cdots \alpha_{r}}m_j^a\right)\;.
%\end{equation}
%\begin{equation}\label{map-expansion}
%Q=e_{\alpha_1}\cdots e_{\alpha_r}:\;\;\langle T|k|Q\rangle=\sum_{i=1}^r\left(\delta^j_{\alpha_i}T_{\alpha_1\cdots \alpha_{i-1} 1 \alpha_{i+1}\cdots \alpha_{r}}-\delta^0_{\alpha_i}T_{\alpha_1\cdots \alpha_{i-1}}{}^j{}_{\alpha_{i+1}\cdots \alpha_{r}}\right)\pi(m_j).
%\end{equation}
\begin{equation}\label{map-expansion}
Q=e_{\alpha_1}\cdots e_{\alpha_r}:\;\;\langle T|k|Q\rangle=\sum_{i=1}^r\left(\delta^j_{\alpha_i}T_{\aa_i 1}-\delta^0_{\alpha_i}T_{\aa_i}{}^j\right)\pi(m_j).
\end{equation}
%where 
%%$\delta^\beta_{\alpha}$ is the Kronecker delta and 
%Greek frame labels 
%%$\alpha_j,\,\alpha,\,\beta$ 
%$\alpha_j$
%run from $0$ to $n-1$. %$\delta_{\alpha\beta}$ is the Kronecker delta. 
%%$\delta_{\alpha\beta}$ equals $1$ if $\alpha=\beta$ and $0$ otherwise.

Now suppose that $(M,g,c)$ is a Lorentzian manifold of dimension $n\geq 3$ with a null congruence (or, equivalently, a null line distribution) $c$ and let %$\mathcal{B}_{r}^{s}$ denote the vector bundle of covariant rank $r$ tensors of boost order $\leq s$along $c$ and 
$\B{r}{s}(M)$ the module of sections of the vector bundle of covariant rank $r$ tensors of boost order $\leq s$ along $c$. Given a non-vanishing section $k$ of $c$ the bilinear map \eqref{def-bilinear-a}-\eqref{def-bilinear-b} applied to each point gives a map
\begin{equation}
    \langle \cdot\vert k\vert\cdot\rangle: \mathcal{B}^{s}_r(M)\times \mathcal{B}^{-s-1}_r(M)\rightarrow \Omega^{0}(c^{\perp}/c),
\end{equation}
where $\Omega^{0}(c^{\perp}/c)$ denotes the space of sections of $c^{\perp}/c.$ From the following theorem, we see that the geometry of $c$ can be partially determined by this bilinear map and the covariant derivative of the tensors involved. 
\begin{theorem}\label{FactorizationTheorem}
If 
%$T\in \mathcal{B}_{r}^{s}(M)$ 
$\bo{T}=s$ on $M$ and $Q\in \mathcal{B}_{r}^{-s-1}(M)$, then
\begin{equation}\label{key}
    (\nabla_{a}T_{c_{1}\cdots c_{r}})Q^{c_{1}\dots c_{r}}=\nabla_{a}k^{b}\,\langle T\vert k \vert Q\rangle_{b}. 
\end{equation}
Furthermore, if 
%$\bo{T}=s$ and 
$\bo{Q}<-s-1$ on $M$ we have $(\nabla_{a}T_{c_{1}\cdots c_{r}})Q^{c_{1}\dots c_{r}}=0$.
\end{theorem}

{\em Note}: For any vector field $X$ on $M$ the vector $X^a\nabla_a k^b$ belongs to $c^\perp$ at each point, and \eqref{key} is the symbolic notation for the collection of rigorous relations
\begin{equation}\label{key-b}
X^a(\nabla_{a}T_{c_{1}\cdots c_{r}})Q^{c_{1}\dots c_{r}}=\pi(X^a\nabla_{a}k^{b})\,\langle T\vert k \vert Q\rangle_{b}.
\end{equation}
%for any vector field $X$. 
%More rigorously, instead of $\langle T\vert k \vert Q\rangle^a\in c^\perp /c$ we should write an arbitrary representative in $c^\perp$ within the right hand side of \eqref{key}; e.g., $\cm(T,Q)^{ab}l_b$ for arbitrary null $l$ with $g(k,l)=1$, see \eqref{def-bilinear-b}. The present notation emphasizes that 
%%the right hand side of \eqref{key} does not depend 
%there is no dependence on the chosen representative since $X^a\nabla_a k^b$ belongs to $c^\perp$ for any vector field $X$ on $M$, and should be read as
%\begin{equation}\label{key-b}
%X^a(\nabla_{a}T_{c_{1}\cdots c_{r}})Q^{c_{1}\dots c_{r}}=\pi(X^a\nabla_{a}k_{b})\,\langle T\vert k \vert Q\rangle^{b} 
%\end{equation}
%for any vector field $X$.
%$\langle T\vert k \vert Q\rangle$ in \eqref{key} should be read as an {\em arbitrary representative} of the corresponding class in $c^\perp/c$; the right hand side of \eqref{key} is then well-defined since for any $X^a$ the vector $X^a\nabla_a k^b$ belongs to $c^\perp$. 

\begin{proof}
Let $\{e_\alpha\}\equiv\{e_0,e_1,e_j\}=\{k,l,m_j\}$
%$\{e_\alpha\}\equiv \{e_{0},\dots e_{n-1} \}=\{k,l,m_2,\dots m_{n-1}$\} 
be a local null frame completing $k$. It is sufficient to prove \eqref{key-b} for $X=e_\beta$ and $Q=e_{\alpha_1}\cdots e_{\alpha_r}$ with $\bw{\aa}=-\bo{Q}\geq s+1$; by property \ref{prop-zero} and equation \eqref{map-expansion} this comes down to showing that ($\delta_{\alpha\beta}\equiv\delta^\beta_\alpha$)
\begin{align}
&\label{nablaT bw geq s+2}\bw{\aa}\geq s+2:\quad \nabla_{\beta}T_{\aa}=0;\\
&\label{nablaT bw s+1}\bw{\aa}=s+1:\quad \nabla_{\beta}T_{\aa}=(\nabla_\beta k)^j\sum_{i=1}^r\left(\delta_{\alpha_i j}T_{\aa_i 1}-\delta_{\alpha_i 0}T_{\aa_i j}\right).
\end{align}
%where  with $\bw{\aa}= s+1$, i.e., 
%\begin{equation}\label{thm2-frame}
%\nabla_{\beta}T_{\alpha_1\cdots \alpha_r}=(\nabla_\beta k)_j\langle T|k|e_{\alpha_1}\cdots e_{\alpha_r}\rangle^j,
%\end{equation}
%and $\nabla_{\beta}T_{\alpha_1\cdots \alpha_r}=0$ when $\bw{\alpha_1\cdots\alpha_r}\geq s+2$. 
%Let $\Gamma^\gamma_{\alpha\beta}=(\nabla_{\beta}e_\alpha)^\gamma$ denote the connection coefficients of the frame. As the frame is rigid [i.e., $\nabla(g(e_\alpha,e_\gamma))=0$] the lowered coefficients $\Gamma_{\gamma\alpha\beta}=g(e_\gamma,\nabla_{\beta}e_\alpha)$ satisfy $\Gamma_{\gamma\alpha\beta}=-\Gamma_{\alpha\gamma\beta}$. 
Let $\Gamma^\alpha_{\beta\gamma}=(\nabla_{\gamma}e_\beta)^\alpha$ denote the connection coefficients of the frame. As the frame is rigid [i.e., $\nabla(g(e_\alpha,e_\beta))=0$] the lowered coefficients $\Gamma_{\alpha\beta\gamma}=g(e_\alpha,\nabla_{\gamma}e_\beta)$ satisfy $\Gamma_{\alpha\beta\gamma}=-\Gamma_{\beta\alpha\gamma}$.
The frame components of $\nabla T$ are given by the well-known expression
%write $\aa\equiv \alpha_1\cdots\alpha_r$ and $\aa_i\gamma\equiv \alpha_1\cdots\alpha_{i-1}\gamma\alpha_{i+1}\cdots \alpha_r$ for convenience. The left hand side of \eqref{thm2-frame} is given by the well-known expression
\begin{equation*}%\label{nablaT-frame}
\nabla_{\beta}T_{\aa}=e_{\beta}(T_{\aa})-\sum_{i=1}^r\Gamma^\gamma_{\alpha_i\beta}T_{\aa_i\gamma}.
\end{equation*}
Since $\bo{T}=s$ we have $e_{\beta}(T_{\aa})=0$  for $\bw{\aa}\geq s+1$. Also, $\bw{\alpha_i}-\bw{\gamma}\geq 2$ for fixed $\alpha_i$ and $\gamma$, with equality only for $\alpha_i=0$ and $\gamma=1$, in which case $\Gamma^1_{0\beta}=\Gamma_{00\beta}=0$; hence a non-zero contribution $\Gamma^\gamma_{\alpha_i\beta}T_{\aa_i\gamma}\neq 0$ to the above expression requires
$$s\geq \bw{\aa_i\gamma}=\bw{\aa}+\bw{\gamma}-\bw{\alpha_i}\geq\bw{\aa}-1\geq s,$$ 
implying %equalities, i.e., 
$\bw{\aa_i\gamma}=s$, $\bw{\alpha_i}=\bw{\gamma}+1$ and $\bw{\aa}=s+1$. 
%with equality only for $\alpha_i=0$ and $\gamma=1$; however, $\Gamma^1_{0\beta}=0$ and so there is no corresponding contribution to the sum in this case. 
Hence %\eqref{nablaT bw geq s+2},
$\nabla_{\beta}T_{\aa}=0$ for $\bw{\aa}\geq s+2$, 
while for $\bw{\aa}=s+1$ we obtain, by $(\nabla_\beta k)^j=\Gamma^j_{0\beta}=\Gamma_{j0\beta}=-\Gamma_{0j\beta}=-\Gamma^{1}_{j\beta}$:%=\Gamma_{j0\beta}=(\nabla_\beta k)_j$:
%non-zero terms require $\bw{\aa_i\gamma}=\bw{\aa}-1=s\Leftrightarrow \bw{\alpha_i}=\bw{\gamma}+1$, and so 
%by $(\nabla_\beta k)^j=\Gamma^j_{0\beta}=-\Gamma^{1}_{j\beta}$ the right hand side of \eqref{nablaT-frame} becomes
%a non-zero term $\Gamma^\gamma_{\alpha_\beta}T_{\aa_i\gamma}$ requires $\bw{\aa_i\gamma}=s$, which occurs for $(\alpha_i,\gamma)=(j,1)$ and $(\alpha_i,\gamma)=(0,j)$
%\begin{align*}
%-\sum_{i=1}^r\Gamma^\gamma_{\alpha_i\beta}T_{\aa_i\gamma}=-\sum_{i=1}^r\left(\delta^j_{\alpha_i}\Gamma^1_{j\beta}T_{\aa_i 1}+\delta^0_{\alpha_i}\Gamma^j_{0\beta}T_{\aa_i j}\right)%=\sum_{i=1}^r\left(\delta^j_{\alpha_i}\Gamma^j_{0\beta}T_{\aa_i 1}-\delta^0_{\alpha_i}\Gamma^j_{0\beta}T_{\aa_i j}\right)
%=(\nabla_\beta k)_j\sum_{i=1}^r\left(\delta^j_{\alpha_i}T_{\aa_i 1}-\delta^0_{\alpha_i}T_{\aa_i}{}^j\right),
%\end{align*}
\begin{align*}
-\Gamma^\gamma_{\alpha_i\beta}T_{\aa_i\gamma}=-\delta^j_{\alpha_i}\Gamma^1_{j\beta}T_{\aa_i 1}-\delta^0_{\alpha_i}\Gamma^j_{0\beta}T_{\aa_i j}%=\sum_{i=1}^r\left(\delta^j_{\alpha_i}\Gamma^j_{0\beta}T_{\aa_i 1}-\delta^0_{\alpha_i}\Gamma^j_{0\beta}T_{\aa_i j}\right)
=(\nabla_\beta k)^j\left(\delta_{\alpha_i j}T_{\aa_i 1}-\delta_{\alpha_i 0}T_{\aa_i j}\right).
\end{align*}
This establishes \eqref{nablaT bw geq s+2} and \eqref{nablaT bw s+1} and thus the theorem.
\end{proof} 

\section{Generic validity of the Kundt Conjecture}\label{sec:Kundtgeneric}
In this section we use the methods developed in the previous section to prove the validity of the Kundt conjecture under a genericity condition on the Ricci tensor (see theorem \ref{Thm:Ricci-generic}).  % and give an easily derived result showing, in arbitrary dimensions, the generic validity of the Kundt conjecture. %In addition we confirm the conjecture in the cases where the Ricci tensor or the Weyl tensor is %non-zero and 
 %of null alignment type III or N. %boost order $< 0.$

Assume that $(M,g,c)$ is a Lorentzian manifold of dimension $n\geq 3$ with a null line distribution $c$ such that for some $N\in \mathbb{N},$ $\nabla^{m}Rm$ is type II or more special w.r.t.\ $c,$ for all $0\leq m\leq N,$ and $k$ is a non-vanishing section of $c.$ For each $r\geq 1$, we let $\mathcal{C}_{r,N}$ be the real vector subspace of $\mathcal{T}_{r}$ generated by the set of covariant rank $r$ tensors which can be obtained from $\{g,g^{-1},Rm,\dots \nabla^{N-1}Rm\}$ through index shuffles, tensor products and %(skew-)symmetrizations and 
contractions in arbitrary indices.

If $p\in M$ consider the subspace of $c_{p}^{\perp}/c_{p}$ given by
\begin{equation}
K_{N}^{c}(p):=\Span\bigcup_{r=1}^{\infty}\langle \mathcal{C}_{r,N}\vert k \vert (\mathcal{B}^{-1}_{r})_{p}\rangle,
\end{equation}
Note that this subspace is independent of the choice of $k\in c_{p}$ and its dimension, 
\begin{equation}
d_{N}^c(p):=\Dim K_{N}^{c}(p),
\end{equation} 
is lower semi-continuous as a function on $M$.  
\begin{claim}\label{KundtOnK}
	The null-line distribution $c$ has the Kundt property on $K_{N}^{c}(p)$, in the sense that 
	\begin{equation}
	[\nabla_{a}k_{b}]X^{a}z^{b}=0,
	\end{equation}
	for all $p\in M,$ $X\in c^{\perp}_{p}$ and $z\in K_{N}^{c}(p).$
\end{claim}
\begin{proof}
	Suppose that $r\geq 1$, $T\in \mathcal{C}_{r,N}$ and $Q\in(\mathcal{B}_{r}^{-1})_{p}.$  By assumption $T$ and $\nabla T$ are type II or more special w.r.t.\ $c$ and therefore $[\nabla_{a}T_{c_{1}\cdots c_{r}}]Q^{c_{1}\cdots c_{r}}\propto k_{a}.$ Together with theorem \ref{FactorizationTheorem} this shows that 
	\begin{equation}
	k_{a}\propto \nabla_{a}k^{b}\langle T\vert k \vert Q\rangle_{b},
	\end{equation}
	which proves the claim.
\end{proof}
\begin{corollary}\label{TopDim}
	If $p\in M$ is a point such that $d_{N}^{c}(p)=n-2,$ then $c$ is Kundt on an open neighborhood of $p.$
\end{corollary}
\begin{proof}
	By lower semi-continuity of $d_N^c$, we can find an open neighborhood $U$ of $p$ such that $d_N^c(q)=n-2$, and therefore $K_{N}^{c}(q)=c_{q}^{\perp}/c_q$, for all $q\in U.$ By claim \ref{KundtOnK}, $c$ is Kundt on $U.$
\end{proof}

Suppose that the trace-free Ricci tensor, $S$, is if type II or more special w.r.t.~$c$ (i.e., $\bo{S}\leq 0\Leftrightarrow S\in \B{2}{0}$). This precisely means that $S$, when regarded as an endomorphism on tangent space, has $c$ as a null eigendirection. %,$k_{[a}S_{b]c}k^c=0$. 
Choose $k\in c$ and complete to a null frame $\{e_0,e_1,e_i\}=\{k,l,m_i\}$. %The boost weight $0$ components $S_{01}$ and $S_{ij}$ are invariant under general null rotations about $k$; we 
We have $S_{ab}k^b=\lambda k^a$, where $\lambda=S_{01}$ is the eigenvalue corresponding to the (unique) timelike generalized eigenspace $E_\lambda$ of $S$. Relative to the frame $\{l,m_i,k\}$ the matrix representation of $S$ takes a lower-triangular block form $\left[\begin{smallmatrix}
\lambda&0&0\\
S_{1i}&S_{ij}&0\\
S_{11}&S_{1i}&\lambda
\end{smallmatrix}\right]$, and from \eqref{map-expansion} we obtain
\begin{equation}\label{map-S}
\langle S|k|km_i\rangle_j = \lambda\delta_{ij}-S_{ij}\,. 
\end{equation}  
Generically, $\lambda$ is not an eigenvalue of the  $(n-2)\times (n-2)$ matrix $[S_{ij}]$; note that this corresponds to $E_\lambda$ being two-dimensional, and implies $\bo{S}=0$. This leads to

\begin{theorem}\label{Thm:Ricci-generic} Let $(M,g,c)$ be a Lorentzian manifold with a null congruence $c$. If at a point $p\in M$ the trace-free Ricci tensor $S$ is of generic type II w.r.t.\ $c$ in the sense that its timelike generalized eigenspace $E_\lambda$ has dimension 2, and if moreover $\bo{\nabla S}\leq 0$, then $c$ is Kundt on an open neighborhood of $p$.  
\end{theorem}

\begin{proof} Applying theorem \ref{FactorizationTheorem} to $T=S\in \B{2}{0}$ and $Q=km_i\in\B{2}{-1}$, and using \eqref{map-S}, we infer 
%and  $\bo{\nabla S}\leq 0$	
%and take \eqref{nablaT bw s+1} with $\beta=0$ or $i$. Using \eqref{map-S} we get
$$
0=X^a\nabla_a S_{0i}=X^a\nabla_a k^j(\lambda\delta_{ij}-S_{ij})%=\nabla_\beta S_{0i}=(S_{01}\delta_{ij}-S_{ij})\nabla_\beta k^j\,,
$$
for all $X^a\in c_p^\perp$. Since $S\in{\cal C}_{2,1}$ and $\lambda$ is not an eigenvalue of $[S_{ij}]$ this implies 
%$\nabla_\beta k_j=0$ for all $j$  and so 
$d_1^c(p)=n-2$, and the theorem follows by corollary \ref{TopDim}.
\end{proof}

\section{The Kundt theorem for Ricci or Weyl type III or N}\label{sec:IIIorN}

In this section we prove the Kundt conjecture in the subcases where the trace-free Ricci or Weyl tensor is of (genuine) null alignment type III or N (in the sense of \cite{OrtPraPrareview}), by repeated use of theorem \ref{FactorizationTheorem}. In fact, we will prove slightly more general results, namely for arbitrary rank 2 symmetric tensors $S_{ab}=S_{ba}$ and symmetric double 2-forms $W_{abcd}=W_{[ab][cd]}=W_{cdab}$. In general, a tensor field of type III or N defines a unique null congruence $c$ along which it has a strictly negative boost order, namely -1 for type III and -2 for type N. In all cases we use a null frame $(e_0,e_1,e_i)=(k,l,m_i)$ with the only restriction that the vector field $k$ generates $c$; we also adopt the commonly used notation $\kappa_i\equiv(\nabla_0 k)_{i}$ and $\rho_{ij}\equiv(\nabla_j k)_{i}$, where $c$ is geodetic iff.~$\kappa_i=0$ for all $i$, and Kundt iff.~$\kappa_i=\rho_{ij}=0$ for all $i$ and $j$. 

Trivially, $k$ is of type III with $\bo{k}=-1$; propositions \ref{prop:Ric-III} and \ref{prop:Weyl-III} below can be seen as generalizations of the following initiatory result.

\begin{proposition}\label{prop:k-III} Let $c$ be a null congruence generated by a null vector field $k$. If $\bo{\nabla k}\leq 0$ then $c$ is geodetic. If moreover $\bo{\nabla\nabla k}\leq 0$ then $c$ is Kundt.
\end{proposition}

\begin{proof} If $\bo{\nabla k}\leq 0$ we have $0=\nabla_0 k_i$, so $c$ is geodetic. If moreover $\bo{\nabla\nabla k}\leq 0$ then 
\begin{equation}\label{k-III}
	0=\nabla_{j}\nabla_{i}k_0=\nabla_j(\nabla k)_{i0}\stackrel{(i)}{=}\nabla_j k_i \nabla_1k_0-\nabla^l k_j{}\nabla_ik_l\stackrel{(ii)}{=}-\rho^l{}_j\rho_{li}
\end{equation}
where we used theorem \ref{FactorizationTheorem} applied to $T=\nabla k$ and $Q=m_i k$ in (i), and $\nabla_1 k_0=0$ in (ii). Contracting \eqref{k-III} on $(i,j)$ gives $\rho^{lj}\rho_{lj}=0$ and thus $\rho_{lj}=0$ for all $l$ and $j$, so $c$ is Kundt. 
\end{proof}
%In the two propositions below we also give the minimal number of covariant derivatives needed to ensure that $c$ is geodetic or Kundt. 
%We will use that, for any fixed $j$:
%\be\label{help-rho}
%\nabla_j k^i\nabla_j k_i=0\quad \Rightarrow \quad \nabla_j k_i=0.
%\ee    

%\noindent Propositions \ref{prop:Ric-III} and \ref{prop:Weyl-III} below may be seen as generalizations of proposition \ref{prop:k-III} to more complicated tensors. 

Let us also note a priori that the boost order of an arbitrary tensor field $T$ along any null congruence $c$ goes up with at most 2 under covariant differentiation ($\bo{T}=s\Rightarrow\bo{\nabla T}\leq s+2$), and with at most 1 if $c$ is geodetic ($\bo{T}=s,\,\nabla_0 k\propto k\Rightarrow\bo{\nabla T}\leq s+1$), as is well known and follows immediately from theorem \ref{FactorizationTheorem}.  

\subsection{Rank 2 symmetric tensors}\label{subsec:rank2symm}

\begin{proposition}\label{prop:Ric-III} Suppose that a rank 2 symmetric tensor $S$ is of type III, and has boost order -1 along the null congruence $c$, $\bo{S}=-1$. If $\bo{\nabla S}\leq 0$ then $c$ is geodetic. If moreover $\bo{\nabla\nabla S}\leq 0$ then $c$ is Kundt.
%the first covariant derivative has (resp.~the first and second covariant derivatives have) negative boost order along $c$ then this null congruence is geodetic (Kundt).
\end{proposition}

\begin{proof} For convenience we write $v_i\equiv S_{1i}=S_{i1}$. % for the boost weight -1 components. 
As $\bo{S}=-1$ we have $v^j v_j>0$. 
Suppose that $\bo{\nabla S}\leq 0$. In particular, this requires $\nabla_0S_{ij}=\nabla_0S_{01}=0$. We can apply theorem \ref{FactorizationTheorem} to $T=S$ and $Q=m_im_j$ \{$Q=kl$\}, taking $X=k$ in \eqref{key-b}; this corresponds to $\beta=0$ and $\aa=ij$ \{$\aa=01$\} in \eqref{nablaT bw s+1}~, and gives
\begin{equation}\label{eq-III-geod}
0=\nabla_0S_{ij}=\kappa_i v_j+\kappa_j v_i,\quad 
0=\nabla_0S_{01}=-\kappa^j v_j.
\end{equation}
Contracting the first equation with $v^j$ and using the second equation yields $\kappa_i(v^jv_j)=0$; hence $\kappa_i=0$ for all $i$, i.e., $c$ is geodetic. 

Suppose now that also $\bo{\nabla\nabla S}\leq 0$. In particular we have $\nabla_k\nabla_jS_{i0}=0$. First we apply theorem \ref{FactorizationTheorem} to $T=\nabla S$, % and $Q=m_jm_ik$; 
taking $\aa=ji0$ and $\beta=k$ in \eqref{nablaT bw s+1}, to obtain 
$$
0=\nabla_k(\nabla S)_{ji0}=\rho_{jk}\nabla_1S_{i0}+\rho_{ik}\nabla_j S_{10}-\rho^l{}_k\nabla_jS_{il}.
$$
For each term in the right hand side we can now apply theorem \ref{FactorizationTheorem} to $T=S$; \eqref{nablaT bw geq s+2} with $\aa=i0$ implies that the first term vanishes, while \eqref{nablaT bw s+1} with $\beta=j$ and $\aa=10$, resp.~$\aa=il$ yields%\footnote{Note that \eqref{eq-III-kundt} is symmetric in $j$ and $k$. This is due to the generalized Ricci identity applied to $S$ and $\bo{S}=-1$, which imply $\nabla_k\nabla_jS_{i0}=\nabla_j\nabla_kS_{i0}$.} 
\begin{align}\label{eq-III-kundt}
%0=-\rho_{ik}\rho^l{}_{j}v_l-\rho^l{}_{k}(\rho_{ij}v_l+\rho_{lj}v_i).
0=-(\rho^l{}_{j}\rho_{ik}+\rho^l{}_{k}\rho_{ij})v_l-\rho^l{}_{k}\rho_{lj}v_i.
\end{align}
Note that this equation is symmetric in $(j,k)$ because $\nabla_k\nabla_jS_{i0}=\nabla_j\nabla_kS_{i0}$, due to the generalized Ricci identity applied to $S$. % and $\bo{S}=-1$. 
Contracting \eqref{eq-III-kundt} with $v^i$ and on $(j,k)$, and putting $w_i=v^l\rho{}_{li}$, yields $2w^iw_i+(\rho^{lj}\rho_{lj})v^iv_i=0$; since $w^iw_i$ and $\rho^{lj}\rho_{lj}$ are non-negative and $v^iv_i>0$, this implies $\rho_{lj}=0$ for all $l$ and $j$, and thus $c$ is Kundt.  
\end{proof}

\begin{proposition}\label{prop:Ric-N} Suppose that a rank 2 symmetric tensor $S$ is of type N, and has boost order -2 along the null congruence $c$, $\bo{S}=-2$. If $\bo{\nabla\nabla S}\leq 0$ then $c$ is geodetic. If moreover $\bo{\nabla\nabla\nabla S}\leq 0$ then $c$ is Kundt.
\end{proposition}

\begin{proof} As $\bo{S}=-2$ we have $S_{11}\neq 0$ and $\bo{\nabla S}\leq 0$. 
Suppose that $\bo{\nabla\nabla S}\leq 0$, such that in particular $\nabla_0\nabla_0 S_{ij}=0$. Again, we start applying theorem \ref{FactorizationTheorem} to $T=\nabla S$, where we now take $\aa=0ij$ and $\beta=0$ in \eqref{nablaT bw s+1}, and then apply the theorem to $T=S$ for each resulting term, making appropriate use of \eqref{nablaT bw geq s+2} and \eqref{nablaT bw s+1}. This leads to
\begin{equation}\label{eq-N-geod}
0=\nabla_0(\nabla S)_{0ij}=-\kappa^m\nabla_mS_{ij}+\kappa_i\nabla_0 S_{1j}+\kappa_j\nabla_0 S_{i1}=\kappa_i\kappa_j S_{11},
\end{equation}
which immediately implies $\kappa_i=0$ for all $i$, i.e., $c$ is geodetic. 

Suppose $\bo{\nabla\nabla\nabla S}\leq 0$ on top of $\bo{\nabla\nabla S}\leq 0$. In particular, $\nabla_l\nabla_k\nabla_j S_{i0}=0$. We apply theorem 2 to $T=\nabla\nabla S$, taking $\aa=kji0$ and $\beta=l$ in \eqref{nablaT bw s+1}, to obtain
\begin{equation*}
0=\nabla_l(\nabla\nabla S)_{kji0}=\rho_{kl}\nabla_1(\nabla S)_{ji0}+\rho_{jl}\nabla_k (\nabla S)_{1i0}+\rho_{il}\nabla_k (\nabla S)_{j10}-\rho^m{}_l\nabla_k(\nabla S)_{jim}.
\end{equation*}
Since $c$ is geodetic we have $\bo{\nabla S}\leq -1$. Hence we can appropriately apply theorem 2 to $T=\nabla S$ for each term of the right hand side. The first term vanishes by \eqref{nablaT bw geq s+2}. Given $\bo{S}=-2$, the generalized Ricci identity applied to $S$ entails $\nabla_k (\nabla S)_{1i0}=\nabla_1 (\nabla S)_{ki0}$, such that the second term also vanishes by \eqref{nablaT bw geq s+2}.
%the same is true for the second term since by the generalized Ricci identity applied to $S$, and given that $\bo{S}=-2$, we have $\nabla_k (\nabla S)_{1i0}=\nabla_1 (\nabla S)_{ki0}=0$. 
Using \eqref{nablaT bw s+1} for the last two terms we obtain
\begin{equation*}
0=\rho_{il}(\rho_{jk}\nabla_1 S_{10}-\rho^m{}_k\nabla_j S_{1m})-\rho^m{}_l(\rho_{jk}\nabla_1 S_{im}+\rho_{ik}\nabla_j S_{1m}+\rho_{mk}\nabla_j S_{i1}).
\end{equation*}
Applying theorem 2 to $T=S$ with appropriate use of \eqref{nablaT bw geq s+2} and \eqref{nablaT bw s+1} eventually yields
\begin{equation}\label{eq-N-kundt}
0=-(\rho_{il}\rho^m{}_k\rho_{mj}+\rho^m{}_l\rho_{ik}\rho_{mj}+\rho^m{}_l\rho_{mk}\rho_{ij})S_{11}.
\end{equation}
Note that this equation is fully symmetric in $(j,k,l)$ because $\nabla_l\nabla_k\nabla_j S_{i0}=\nabla_k\nabla_l\nabla_j S_{i0}=\nabla_l\nabla_j\nabla_k S_{i0}$, due to the generalized Ricci identity applied to $S$ and $\nabla S$. Contracting \eqref{eq-N-kundt} with $\rho^{il}$ and on $(j,k)$, and putting $A_{ij}=\rho^m{}_i\rho_{mj}$, yields $2A^{ij}A_{ij}+\rho^{ij}\rho_{ij}=0$;  
%Taking $j$ fixed and $k=l=j$ this gives $\rho_{ij}\rho^m{}_j\rho_{mj}=0$ for all $i$; 
hence $\rho_{ij}=0$ for all $i$ and $j$, so $c$ is Kundt. 
\end{proof}

\subsection{Symmetric double 2-forms}\label{subsec:symm 2-form}

Symmetric double 2-forms lead to analogous propositions. One starts by deriving relevant surrogates for equations \eqref{eq-III-geod}-\eqref{eq-N-kundt}, successively using theorem 2 in the same way as in the respective corresponding cases for rank 2 symmetric tensors; we will merely provide the results of these calculations. From here on, the confirmation of the geodetic and Kundt properties of $c$ is slightly more complex.

%For symmetric double 2-forms one arrives at surrogates of equations \eqref{eq-III-geod}-\eqref{eq-N-kundt} by similar calculations, making consecutive use of theorem 2 in the same fashion as in the corresponding case for rank 2 symmetric tensors. In the proofs below we start with the mere results of these calculations; from there, the proofs of the geodetic and Kundt properties for $c$ are slighlty more elaborate, but the results are analogous.  

\begin{proposition}\label{prop:Weyl-III} Suppose that a symmetric double 2-form $W$ is of type III, and has boost order -1 along the null congruence $c$, $\bo{W}=-1$. If $\bo{\nabla W}\leq 0$ then $c$ is geodetic. If moreover $\bo{\nabla\nabla W}\leq 0$ then $c$ is Kundt.
	%the first covariant derivative has (resp.~the first and second covariant derivatives have) negative boost order along $c$ then this null congruence is geodetic (Kundt).
\end{proposition}

\begin{proof} For convenience we write $\Psi_{ijk}\equiv W_{1ijk}$ and $\Psi_i\equiv W_{101i}$; given $\bo{W}=-1$ these components do not all vanish. Suppose that $\bo{\nabla W}\leq 0$. Applying theorem~\ref{FactorizationTheorem} to $T=W$ we obtain  
%If the boost order of $\nabla W$ along $c$ is negative, we have
%\begin{align}
%\label{III-ijkl}&0=\nabla_0 W_{ijkl}=\nabla_0k_i\Psi_{jkl}-\nabla_0k_j\Psi_{ikl}+\nabla_0 k_k\Psi_{lij}-\nabla_0k_l\Psi_{kij},\\
%\label{III-01ij}&0=\nabla_0 W_{01ij}=\Psi_{mij}\nabla_0k^m+\Psi_i\nabla_0k_j-\Psi_j\nabla_0k_i,\\
%\label{III-0i1j}&0=\nabla_0 W_{0i1j}=-\Psi_{jmi}\nabla_0 k^m -\Psi_j\nabla_0 k_i,\\
%\end{align}
\begin{align}
\label{III-ijkl}&0=\nabla_0 W_{ijkl}=\kappa_i\Psi_{jkl}-\kappa_j\Psi_{ikl}+\kappa_k\Psi_{lij}-\kappa_l\Psi_{kij}\;;\\
\label{III-01ij}&0=\nabla_0 W_{01ij}=\kappa^m\Psi_{mij}-\kappa_i\Psi_j+\kappa_j\Psi_i\;;\\
\label{III-0i1j}&0=\nabla_0 W_{0i1j}=-\kappa^m\Psi_{jmi} -\kappa_i\Psi_j\;.
\end{align}
Assume $c$ is not geodetic, i.e., $\kappa^i\kappa_i\neq 0$. Contracting \eqref{III-0i1j} with $\kappa^i$ gives $\Psi_j=0$ for all $j$; from \eqref{III-ijkl}-\eqref{III-0i1j} we then obtain 
$$
0=\kappa^i\nabla_0 W_{ijkl}=(\kappa^i\kappa_i)\Psi_{jkl}
$$
and thus $\Psi_{jkl}=0,$ for any $j,k,l$, which is a contradiction. 
%hence all b.w.\ -1 components of the Weyl tensor would vanish, contradictory to the type III assumption. 
Hence $c$ is geodetic. 

Suppose that also $\bo{\nabla\nabla W}\leq 0$. In particular, $\nabla_m\nabla_l W_{ijk0}=\nabla_m\nabla_l W_{i010}=0$ (which are symmetric in $(l,m)$ as before). By applying theorem 2 to $T=\nabla W$ and consecutively to $T=W$ these conditions become 
\begin{align}\label{W-III-a0}
%\begin{equation}\label{W-III-a0}
&\begin{aligned}
0=\;&(\rho^q{}_l\rho_{jm}+\rho^q{}_m\rho_{jl})\Psi_{ikq}-(\rho^q{}_l\rho_{im}+\rho^q{}_m\rho_{il})\Psi_{jkq}-(\rho^q{}_l\rho_{km}+\rho^q{}_m\rho_{kl})\Psi_{qij}\\
&+\rho^q{}_l\rho_{qm}\Psi_{kij}+(\rho_{kl}\rho_{im}+\rho_{km}\rho_{il})\Psi_{j}-(\rho_{kl}\rho_{jm}+\rho_{km}\rho_{jl})\Psi_{i}\;;
\end{aligned}\\
%\end{equation}
%\begin{equation}
&\label{W-III-b0}
0=(\rho^q{}_l\rho^r{}_{m}+\rho^q{}_m\rho^r{}_{l})\Psi_{qir}-2(\rho^q{}_l\rho_{im}+\rho^q{}_m\rho_{il})\Psi_q+\rho^q{}_l\rho_{qm}\Psi_{i}\;.
%\end{equation}
\end{align}
Take an arbitrary but fixed index $l^*\in\{2,\ldots,n-1\}$, put both $l$ and $m$ equal to $l^*$, and write $\rho_i\equiv\rho_{il^*}$ for convenience. Then \eqref{W-III-a0} and \eqref{W-III-b0} specify to
\begin{align}
\label{W-III-a}&0=2\rho^q(\rho_j\Psi_{ikq}-\rho_i\Psi_{jkq}-\rho_k\Psi_{qij})+\rho^q\rho_q\Psi_{kij}+2\rho_k(\rho_i\Psi_j-\rho_j\Psi_i)\;;\\
\label{W-III-b}&0=2\rho^q\rho^r\Psi_{qir}-4\rho_i(\rho^q\Psi_q)+\rho^q\rho_q\Psi_i\;.
\end{align}
Assume $\rho^k\rho_k\neq 0$. Contracting \eqref{W-III-a} with $\rho^j\rho^k$ and using \eqref{W-III-b} yields 
$
-\tfrac{3}{2}(\rho^k\rho_k)^2\Psi_i=0
$
and hence $\Psi_i=0$ for all $i$. Herewith \eqref{W-III-b} reduces to $\rho^q\rho^r\Psi_{qir}=0$. Now separate contractions of \eqref{W-III-a} with $\rho^k$ and $\rho^j$ lead to $(\rho^k\rho_k)\rho^q\Psi_{qij}=0$ and $(\rho^j\rho_j)(2\rho^q\Psi_{ikq}+\rho^q\Psi_{kiq})=0$. Interchanging also $i$ and $k$ in the last equation this implies $\rho^q\Psi_{qij}=0$ and $\rho^q\Psi_{ijq}=0$, but then we obtain $\Psi_{ijk}=0$ for all $i,j,k$ by \eqref{W-III-a}. Hence all boost weight -1 components of $W$ vanish, which is a contradiction. We conclude that $\rho^k\rho_k= 0$ and hence $\rho_k=\rho_{kl^*}=0$ for all $k$ and any $l^*$, so $c$ is Kundt.   
\end{proof}

\begin{proposition}\label{prop:Weyl-N} Suppose that a symmetric double 2-form $W$ is of type N, and has boost order -2 along the null congruence $c$, $\bo{W}=-2$. If $\bo{\nabla\nabla W}\leq 0$ then $c$ is geodetic. If moreover $\bo{\nabla\nabla\nabla W}\leq 0$ then $c$ is Kundt.
\end{proposition}

\begin{proof} For convenience we write $\Psi_{ij}\equiv W_{1i1j}=\Psi_{ji}$. Since $\bo{W}=-2$ these components do not all vanish, and $\bo{\nabla W}\leq 0$. Suppose that $\bo{\nabla\nabla W}\leq 0$. In particular, $\nabla_0\nabla_0 W_{0i1j}=0$, and by applying theorem  \ref{FactorizationTheorem} successively to $T=\nabla W$ and $T=W$ we infer that
\begin{equation}\label{W-N-geod}
0=\nabla_0(\nabla W)_{00i1j}=2\kappa_i\kappa^k\Psi_{kj}-\kappa^k\kappa_k\Psi_{ij}.
\end{equation}
%where we applied theorem \ref{FactorizationTheorem} to $T=\nabla W$ and $T=W$. 
Contracting \eqref{W-N-geod} with $\kappa^i$ gives $0=(\kappa^i\kappa_i)\kappa^k\Psi_{kj}$, implying $\kappa^k\Psi_{kj}=0$, and then we obtain $\kappa^k\kappa_k=0$ from \eqref{W-N-geod} itself. Hence $\kappa_k=0$ for all $k$, so $c$ is geodetic.
%to $\kappa^k\kappa_k\Psi_{ij}=0$. Since not all $\Psi_{ij}$ are zero we thus find $\kappa_k=0$, i.e., $c$ is geodetic. As a consequence, $\bo{\nabla W}\leq -1$.
	
Suppose $\bo{\nabla\nabla\nabla W}\leq 0$ on top of $\bo{\nabla\nabla W}\leq 0$. As before, take an arbitrary but fixed index $l^*\in\{2,\ldots,n-1\}$, put $l$, $m$ and $n$ equal to $l^*$, and write $\rho_i\equiv\rho_{il^*}$ for convenience. In particular we have $\nabla_{l^*}\nabla_{l^*}\nabla_{l^*}W_{ijk0}=\nabla_{l^*}\nabla_{l^*}\nabla_{l^*}W_{i010}=0$. By successive application of theorem 2 to $T=\nabla\nabla W$, $T=\nabla W$ and $T=W$ these conditions become
\begin{align}
\label{W-N-kundt-a}\rho^q\rho_q(\rho_i\Psi_{jk}-\rho_j\Psi_{ik})&=2\rho_k\rho^q(\rho_i\Psi_{jq}-\rho_j\Psi_{iq}),\\
\label{W-N-kundt-b}\rho^q\rho_q\rho^k\Psi_{ik}&=2\rho_i(\rho^j\rho^k\Psi_{jk})
\end{align}
Assume $\rho^k\rho_k\neq 0$. Contracting \eqref{W-N-kundt-a} with $\rho^j\rho^k$ implies $\rho^q\rho_q\rho_k\Psi_{ik}=\rho_i\rho^q\rho^r\Psi_{qr}$, and thus $\rho_k\Psi_{ik}=0$ by comparison with \eqref{W-N-kundt-b}; however, contracting now \eqref{W-N-kundt-a} with $\rho^i$ leads to $\Psi_{jk}=0$ for all $j$ and $k$, which is a contradiction. Hence $\rho_k=\rho_{kl^*}=0$ for all $k$ and any $l^*$, which proves that $c$ is Kundt.
\end{proof}

\section{Proof of the Kundt theorem in dimension 3}\label{sec:Kundt3}
Here we give a proof of the null-alignment Kundt theorem in dimension 3 using the notation and results of previous sections. In our current setting we see that $0\leq d_{1}^c(p)\leq1,$ for all $p\in M,$ and any $c$, and the information of the Riemann tensor is contained in the Ricci scalar and the trace-free Ricci tensor $S$. 
\begin{theorem}
	Suppose that $(M,g,c)$ is a three-dimensional Lorentzian manifold with null congruence $c$ such that the tensors $\nabla^{m}S$ are of aligned type II or more special  
	%algebraically special 
	for $m\leq 3,$ and let $\Gamma\subset M$ be the set of points for which $S$ is non-zero. Then $c$ has the Kundt property on the closure $\bar{\Gamma}.$
\end{theorem}

\begin{proof}
	Suppose that $c$ does not have the Kundt property at a point $p.$ By continuity we can find a neighborhood $U$ of $p$ on which this holds. Corollary \ref{TopDim} implies that $\langle S\vert k \vert Q\rangle=0,$ for all rank 2 tensors $Q$ with $\bo{Q}=-1$;
	%of boost-order $\leq -1$ w.r.t.\ $c$; 
	by \eqref{map-S} and the trace-free property of $S$ it follows that $\bo{S}\leq -1$ on $U$. % $S$ is of boost-order $\leq-1$ w.r.t.\ $c$ on $U.$ 
	If each neighborhood of $p$ contains a point for which $\bo{S}=-1$
	%$S$ is of boost-order $-1$ w.r.t. $c$ 
	then lower semi-continuity of boost order together with proposition \ref{prop:Ric-III} implies that $c$ is Kundt, which gives a contradiction. Therefore $\bo{S}\leq -2$
	%$S$ is of boost-order $\leq -2$ w.r.t $c$ 
	on some neighborhood $V\subset U$ of $p$. By the same procedure and proposition \ref{prop:Ric-N} one shows that $p$ cannot be approximated by points for which $\bo{S}=-2$, 
	%is of boost-order $-2$ w.r.t. $c,$ 
	and hence $S$ vanishes on some neighborhood of $p$, showing that $p\notin \bar{\Gamma}.$
\end{proof}

\section{Proof of the Kundt theorem in dimension 4}\label{sec:Kundt4}

In this section we provide a new proof of the Kundt theorem in dimension 4. First we present a broader result on the partial map $\langle W\vert k \vert \cdot\rangle$ for Weyl-like tensors $W$. This entails theorem \ref{NonConformallyFlatKundtChar} and corollary \ref{cor:WeylII-D}, which covers Weyl-Petrov types II and D. In theorem \ref{ConformallyFlatTachyonic} we focus on conformally flat spacetimes for which the trace-free Ricci tensor is of type II or D w.r.t.~$c$. For the proof we again make extensive use of theorem \ref{FactorizationTheorem}; the introduced technique may well apply to more general tensors in later investigations on the Kundt conjecture in higher dimensions. However, a streamlined proof specific to dimension 4 is also given in remark \ref{rem:4Dproof}, which may be directly compared with the more tedious calculations used in \cite{CharBySpi}. The culmination point is theorem \ref{thm:culmination4d}, which assembles the relevant results of the present and previous sections.

\begin{proposition}\label{BilWeylOnto}
	Let $(V,g,c)$ be a four-dimenional Lorentzian vector space with a null-line $c$ and $W$ a non-zero Weyl-like tensor, i.e., $W$ satisfies
	
	(i) %$W_{abcd}=-W_{bacd}=-W_{abdc}=W_{cdab}$;
	$W_{abcd}=W_{[ab][cd]}=W_{cdab}$;
	
	(ii) $W_{abcd}+W_{acdb}+W_{adbc}=0$;
	
	(iii) $W_{abc}{}^b=0$.\\
	If $c$ is aligned with  $W$ and $\bo{W}=s$ ($-2\leq s\leq 1$) 
	%If $W$ is of boost order $s\leq 1$ w.r.t. $c$, 
	then, for any $k\in c$, the map
	\begin{equation}\label{map-Weyl}
	\langle W\vert k \vert \cdot\rangle : \mathcal{B}_{4}^{-s-1}\rightarrow c^{\perp}/c,
	\end{equation}
	is surjective. 
\end{proposition}
\begin{proof}
	Take any $k\in c$. For convenience we complete this vector to a complex null frame $(e_0,e_1,e_{\mm},e_{\mmc})=(k,l,m,\overline{m})$, where the complex conjugate null vectors are associated to the spatial vectors $e_2$ and $e_3$ of a real null frame by $m=\frac{1}{\sqrt{2}}(e_2-ie_3)$ and $\overline{m}=\frac{1}{\sqrt{2}}(e_2+ie_3)$. Referring to Newman-Penrose notation we define the  complex scalars $\Psi_{i}$ ($0\leq i\leq 4$) by
	\begin{equation}
	\begin{aligned}
	&\Psi_0=W_{abcd}k^am^bk^cm^d,\quad \Psi_1=W_{abcd}k^al^bk^cm^d,\quad \Psi_2=W_{abcd}k^am^bl^c\overline{m}^d,\\
	&\Psi_3=W_{abcd}l^ak^b\overline{m}^c l^d,\quad \Psi_4=W_{abcd}l^a\overline{m}^bl^c\overline{m}^d.
	\end{aligned}
	\end{equation}
	%which are the usual Newman-Penrose scalars. 
	Since $c$ is a principal null direction of $W$ we have $\Psi_0=0$. It is well known that, due to the properties (i)-(iii), the frame components of $W$ of boost weight $m$ are real linear combinations of $\Psi_{2-m}$ and $\overline{\Psi_{2-m}}$; hence,
	$$
	\bo{W}=s\;\;(-2\leq s\leq 1)\quad \Leftrightarrow\quad \Psi_{i}=0\;\;\textrm{for}\;\;0\leq i< 2-s\;\;\textrm{and}\;\;\Psi_{2-s}\neq 0.
	$$
	For $-1\leq s\leq 2$ we define $Q_s\in{\cal B}_4^{-s-1}$ by
	\begin{equation}
	Q_{-2}=l\overline{m}m\overline{m},\quad Q_{-1}=kmlm,\quad Q_{0}=kmm\overline{m},\quad Q_1=kmkm.
	\end{equation}
	Applying \eqref{map-expansion} to $T=W$ and $Q=Q_s$, and using the properties (i)-(iii) of the Weyl tensor, in particular the frame expansion of the tracefree property (iii) reading
	\begin{equation}\label{tracefree-frame}
	W_{\alpha 0\beta1}+W_{\alpha 1\beta0}+W_{\alpha \mm\beta\mmc}+W_{\alpha \mmc\beta \mm}=0
	\end{equation}
	for any $\alpha,\beta\in\{0,1,\mm,\mmc\}$, one readily finds\footnote{
		%The frame expansion of the tracefree property (iii) reads $W_{\alpha 0\beta1}+W_{\alpha 1\beta0}+W_{\alpha \mm\beta\mmc}+W_{\alpha \mmc\beta \mm}=0$ for any $\alpha,\beta\in\{0,1,\mm,\mmc\}$; for
		For $s=-2$, $s=-1$, $s=0$, and $s=1$ one uses $W_{1\mm 1\mmc}=0$, $-W_{1\mmc\mm\mmc}=W_{10\mmc1}$, $-W_{\mmc \mm\mm\mmc}=W_{0\mm 1\mmc}+W_{0\mmc 1\mm}$ and $W_{0\mm\mm\mmc}=W_{010\mm}$, which are obtained by (i) and applying \eqref{tracefree-frame} to $\alpha=\beta=1$, $\alpha=1$ and $\beta=\mmc$, $\alpha=\mmc$ and $\beta=\mm$, and $\alpha=0$ and $\beta=\mm$, respectively; for $s=0$ one also uses $W_{01\mm\mmc}=W_{0\mm 1\mmc}-W_{0\mmc 1\mm}$, which follows from (i) and (ii).}
	$$
	\langle W|k|Q_s\rangle = (s+3)\Psi_{2-s}\,\pi(m).
	$$
	Together with the complex conjugates $\langle W|k|\overline{Q_s}\rangle = (s+3)\overline{\Psi_{2-s}}\,\pi(\overline{m})$ this shows that the map \eqref{map-Weyl} is surjective for each $s$. 
\end{proof}

We call a null congruence {\em Robinson-Trautman} if it is non-twisting (and therefore geodetic) and non-shearing but expanding; if the congruence is generated by a non-vanishing vector field $k$ then it is Robinson-Trautman iff.~$0\neq \pi(X^a\nabla_a k^b)\propto \pi(X)^b$ for all vector fields $X\in c^{\perp}$. Recall that a null congruence generated by $k$ is Kundt if it is non-twisting, non-shearing and non-expanding, i.e., iff.~$\pi(X^a\nabla_a k^b)=0$ for all $X\in c^{\perp}$. From theorem \ref{FactorizationTheorem} %[see \eqref{key-b}] 
and proposition \ref{BilWeylOnto} we immediately infer: 
%$$
%\pi(X^a\nabla_a k^b)=0\;\{\propto \pi(X)^b\}.
%$$

\begin{theorem}\label{NonConformallyFlatKundtChar}
	Suppose that $(M,g,c)$ is a four-dimensional, non-conformally flat Lorentzian manifold with a null-congruence $c$ such that a Weyl-like tensor $W$ has constant boost order $s\leq1$ along $c$. Letting $k$ be a non-vanishing vector field generating the null-congruence, then $c$ is Kundt $\{$Robinson-Trautman$\}$ iff.~$\{\exists\, \alpha\in C^{\infty}(M)\, :\}$
	\begin{equation}
	X^{a}(\nabla_{a}W_{bcde})Q^{bcde}=0\, \{=\alpha \pi (X)^{a}\langle W\vert k\vert Q\rangle_{a}\},
	\end{equation}
	for all $X\in c^{\perp}$ and $Q\in \B{4}{-s-1}(M)$.
\end{theorem}

%\begin{proof} A null congruence $c$, generated by $k$, is Kundt $\{$Robinson-Trautman$\}$ iff.~for any vector field $X\in c^{\perp}$, 
%$$
%\pi(X^a\nabla_a k^b)=0\;\{\propto \pi(X)^b\}.
%$$
%The result follows now immediately from theorem \ref{FactorizationTheorem} [see \eqref{key-b}] and proposition \ref{BilWeylOnto}.
%\end{proof}

\begin{corollary}\label{cor:WeylII-D} Suppose that a Weyl-like tensor $W$ defined on a spacetime $(M,g)$ is of Petrov type II or D everywhere, with constant boost order $s=0$ along the multiple principal null direction $c$ ($\bo{W}=0$). If moreover $\bo{\nabla W}\leq 0$ then $c$ is Kundt.
\end{corollary}

%\begin{proof} For any $X\in c^\perp=\B{1}{0}$ and $Q\in\B{4}{-1}$ the function $f=X^{a}(\nabla_{a}W_{bcde})Q^{bcde}\in C^{\infty}(M)$ has $\bo{f}\leq \bo{X}+\bo{\nabla W}+\bo{Q}=-1$; thus $f=0$ and the result follows from theorem \ref{NonConformallyFlatKundtChar}. 
%\end{proof}

Let us state a convention that will be useful in expressing the theorems below. We shall say that a collection of tensors $\{T_{\lambda}\}_{\lambda\in I}$ is uniformly of type $D$ w.r.t. $c$ at a point $p\in M$ if there exists a null line $\tilde{c}_{p}\subset T_{p}M$ different from $c_{p}$ such that for each $\lambda\in I$, $(T_{\lambda})_{p}$ is either zero or algebraically special w.r.t. $c_{p}$ and $\tilde{c}_{p}$, and for some $\tau\in I,$ $(T_{\tau})_{p}\neq 0.$

\begin{theorem}\label{ConformallyFlatTachyonic}
	Let $(M,g,c)$ be a conformally flat Lorentzian manifold with a null line distribution $c$ such that
	\begin{enumerate}[a)]
		\item the trace-free Ricci tensor $S$ is everywhere non-zero and of boost-order zero w.r.t. $c$;
		\item $\nabla^{m}Rm$ is algebraically special w.r.t. $c$, for $0\leq m\leq 3,$
		\item $d_{3}^{c}\equiv 1$ on $M$,
	\end{enumerate}
	Then the following holds:
	\begin{enumerate}[i)]
		\item If $\{S,\nabla S\}$ is nowhere uniformly of type $D$ w.r.t. $c,$ then $c$ is Kundt on $M.$
		\item If $\{S,\nabla S\}$ is uniformly of type $D$ w.r.t. $c$ everywhere, then  for
		each point of $p\in M$ there exists a neighborhood $U$, an interval $I=(-\epsilon,\epsilon)$, a three-dimensional locally maximally symmetric Lorentzian manifold $(N,h)$ and a strictly positive smooth function $f:I\rightarrow \mathbb{R}$ such that $U$ is isometric to  $I\times N$ with the warped product metric
		\begin{equation}\label{warp}
		dx^{2} + fh,
		\end{equation}
		where $x$ is a coordinate on $I.$
	\end{enumerate}
\end{theorem}
\begin{proof}
Let  $\{k,l,m_2,m_3\}$ be a null-frame on an open set $U$ such that $k$ is aligned with $c$ and $\langle\pi(m_2)\rangle=K_{3}^{c}$ at each point.
	Then,
	\begin{equation}
	0=\langle S\vert k \vert km_3\rangle_{3}=-S_{33}+S_{01},
	\end{equation}
	and
	\begin{equation}
	0=\langle S\vert k \vert km_2\rangle_{3}=-S_{23},
	\end{equation}
	and hence $S_{22}\neq S_{01}$ everywhere on $U$ since otherwise $S$ would not be of type $II$ or $D$. Performing a null-rotation of the form
	\begin{equation}
	k\mapsto k^{\prime}=k,\quad l\mapsto l^{\prime}=l -\frac{1}{2}z^{2}k + z m_{2},\quad m_{2}\mapsto m_{2}^{\prime}=m_{2}-z k,\quad m_3\mapsto m_{3}^{\prime}=m_{3},
	\end{equation}
	gives
	\begin{equation}
	S_{1^{\prime}2^{\prime}}=S_{12} -z (S_{01}-S_{22}).
	\end{equation}
	Letting $z = \frac{S_{12}}{S_{01}-S_{22}}$, we obtain a new frame $\{k,l,m_{2},m_{3}\}$ such that $\langle\pi(m_{2})\rangle=K_{3}^{c}$ and $S_{12}=0$ on $U.$ Any aditional null-rotation of the frame about $k$ which fixes $m_2$  will  leave the components $S_{13}$ and $S_{11}$ unchanged and therefore $S$ is of type $II$ iff. $(S_{11})^2+(S_{13})^2$ is everywhere non-zero and $S$ is of type $D$ iff. $S_{11}=S_{13}\equiv 0.$
	
	Using that $\langle \nabla S\vert k\vert \mathcal{B}_{3}^{-1}\rangle\subset \langle \pi(m_2)\rangle, $ we obtain the following list of identities: 
	\begin{align} 0=&\langle\nabla S \vert k \vert km_{3}m_{3}\rangle_{3}=-\nabla_{3}S_{33}+2\nabla_{0}S_{13},\label{DSRel1}\\
	0=&\langle\nabla S \vert k \vert m_{3}km_{3}\rangle_{3}=-\nabla_{3}S_{33}+\nabla_{1}S_{03}+\nabla_{3}S_{01},\\
	0=&\langle\nabla S \vert k \vert lkk\rangle_{3}=-2\nabla_{1}S_{03},\\
	0=&\langle\nabla S \vert k \vert kkl\rangle_{3}=-\nabla_{3}S_{01}-\nabla_{0}S_{13},\label{DSRel2}\\
	0 =& \langle \nabla S\vert k \vert km_2m_3\rangle_{3}=-\nabla_{3}S_{23} + \nabla_{0}S_{21},\label{DSRel323}\\
	0=& \langle \nabla S\vert k \vert m_2km_3\rangle_{3}= -\nabla_{2}S_{33} + \nabla_{2}S_{01},\label{DSRel233}\\
	0=& \langle \nabla S\vert k \vert m_3km_2\rangle_{3}=-\nabla_{3}S_{32} + \nabla_{1}S_{02},\label{DSRel332}\\
	0=&\langle \nabla S\vert k \vert km_2m_2\rangle_{3} =-\nabla_{3}S_{22},\label{DSRel322}\\
	0=&\langle \nabla S\vert k \vert m_2m_2k\rangle_{3} =-\nabla_{2}S_{23}.\label{DSRel223}
	\end{align}
	Using that $\nabla S$ is algebraically special and  relations \eqref{DSRel323}, \eqref{DSRel332} and \eqref{DSRel223}  we proceed to show that $\nabla m_2$ is algebraically special and satisfies \begin{equation}\label{m2Rel}
	\nabla_{2}(m_{2})_{3}=0,\quad\nabla_{1}(m_2)_{0}=\nabla_{0}(m_2)_{1}=\nabla_{3}(m_2)_{3}.\end{equation}
	\begin{align}
	0=&\nabla_{0}S_{20}=(-S_{10}+S_{22})\nabla_{0}(m_{2})_{0},\\
	0=&\nabla_{0}S_{23}=(-S_{33}+S_{22})\nabla_{0}(m_{2})_{3},\label{0m3}\\
	0=&\nabla_{2}S_{23}=(-S_{33}+S_{22})\nabla_{2}(m_{2})_{3},\\
	0=&-\nabla_{3}S_{23}+\nabla_{1}S_{02}=(S_{33}-S_{22})(\nabla_{3}(m_{2})_{3}-\nabla_{1}(m_{2})_{0}),\\
	0=&-\nabla_{3}S_{23}+\nabla_{0}S_{12}=(S_{33}-S_{22})(\nabla_{3}(m_{2})_{3}-\nabla_{0}(m_{2})_{1}),
	\end{align}
	where we used \eqref{0m3} in the last equation. This gives the stated properties of $m_2$.
	
	Combining eqs. \eqref{DSRel1}-\eqref{DSRel2}, shows that
	\begin{equation}\label{DSRel333}
	\nabla_3S_{33}=\nabla_{1}S_{03}=\nabla_{0}S_{13}=\nabla_{3}S_{01}=0.
	\end{equation}
	From eqs. \eqref{DSRel233} and \eqref{DSRel333}, we find that
	\begin{equation}\label{S13Rel}
	0\equiv\nabla_{X}S_{33}-\nabla_{X}S_{01}=X(S_{33}-S_{01}) -2S_{a3}\nabla_{X}m_{3}^{a}+S_{a1}\nabla_{X}k^{a}+S_{0a}\nabla_{X}l^{a} =3S_{13}\nabla_{X}k_{3},
	\end{equation}
	for all $X\in c^{\perp}.$ 
	
	Assume from now on that $S_{13}\equiv 0$ on $U.$ Then, 
	\begin{equation}
	\nabla_{1}S_{10}-\nabla_{1}S_{33}=l(S_{10}-S_{33}) - S_{10}\nabla_{1}l_{0}-S_{10}\nabla_{1}k_{1} + -2S_{33}\nabla_{1}(m_{3})_{3}=0.
	\end{equation}
	
	With the identities above, we derive several relations by using $\langle \nabla\nabla S\vert k\vert \mathcal{B}_{4}^{-1}\rangle\subset\langle\pi(m_2)\rangle:$
	\begin{align}
	0=&\langle\nabla\nabla S\vert k \vert m_{3}m_{3}km_{3}\rangle_{3}
	=-\nabla_{3}\nabla_{3}S_{33}+\nabla_{1}\nabla_{3}S_{03}+\nabla_{3}\nabla_{1}S_{03}+\nabla_{3}\nabla_{3}S_{01}\nonumber\\
	&\hspace{4cm}=-3\nabla_{3}S_{13}\nabla_{3}k_{3},\label{DDSRela}\\
	0=&\langle\nabla\nabla S\vert k\vert m_2m_{3}km_{3}\rangle_{3}=-\nabla_{2}\nabla_{3}S_{33}+\nabla_{2}\nabla_{1}S_{03}+\nabla_{2}\nabla_{3}S_{01}\nonumber\\&\hspace{4cm}=-3\nabla_{3}S_{13}\nabla_{2}k_{3},\\
	0=&-\nabla_{0}\nabla_{3}S_{33}+\nabla_{0}\nabla_{3}S_{01}=-3\nabla_{3}S_{13}\nabla_{0}k_{3},
	\end{align}
	\begin{align}
	0=&\langle\nabla\nabla S\vert k \vert m_{3}m_{2}km_{3}\rangle_{3} = -\nabla_{3}\nabla_{2}S_{33}+\nabla_{1}\nabla_{2}S_{03}+\nabla_{3}\nabla_{2}S_{01}\nonumber\\
	&\hspace{4cm}=-3\nabla_{2}S_{13}\nabla_{3}k_{3},\\
	0=&\langle\nabla\nabla S\vert k \vert m_{2}m_{2}km_{3}\rangle_{3} =-\nabla_{2}\nabla_{2}S_{33}+\nabla_2\nabla_2S_{01}\nonumber\\ &\hspace{4cm}=
	-3\nabla_{2}S_{13}\nabla_{2}k_{3},\\
	0=& -\nabla_{0}\nabla_{2}S_{33}+\nabla_{0}\nabla_{2}S_{01}=-\nabla_{2}S_{13}\nabla_{0}k_{3},
	\end{align}
	\begin{align}
	0=&\langle\nabla\nabla S\vert k \vert m_{3}m_{2}km_{2}\rangle_{3} =-\nabla_{3}\nabla_{2}S_{23}+\nabla_{1}\nabla_{2}S_{02}=-\nabla_{2}S_{12}\nabla_{3}k_3,\\
	0=&\langle\nabla\nabla S\vert k \vert m_{2}m_{2}km_{2}\rangle_{3} =-\nabla_{2}\nabla_{2}S_{23}=-\nabla_{2}S_{12}\nabla_{2}k_{3},\\
	0=&\nabla_{0}\nabla_{2}S_{23}=\nabla_{2}S_{12}\nabla_{0}k_{3},
	\end{align}
	\begin{align}
	0=&\langle\nabla\nabla S\vert k \vert m_{3}km_{2}m_2\rangle_{3} =-\nabla_{3}\nabla_{3}S_{22}+\nabla_{1}\nabla_{0}S_{22}=-\nabla_{1}S_{22}\nabla_{3}k_3,\\
	0=&\langle\nabla\nabla S\vert k \vert m_{2}km_{2}m_{2}\rangle_{3} =-\nabla_{2}\nabla_{3}S_{22}=-\nabla_{1}S_{22}\nabla_{2}k_{3},\\
	0=&\nabla_{0}\nabla_{3}S_{22}=\nabla_{1}S_{22}\nabla_{0}k_{3}.\label{DDSRelb}
	\end{align}
	In addition we have the identities
	\begin{equation}\label{S11Rel}
	\nabla_{0}S_{13}=S_{11}\nabla_{0}k_{3},\quad \nabla_{2}S_{13}=S_{11}\nabla_{2}k_{3},\quad \nabla_{3}S_{13}=S_{11}\nabla_{3}k_{3}.
	\end{equation}
	
	Now we are ready to prove  assertion $i)$. Suppose there exists a point $p\in M $ such that $c$ does not have the Kundt property and fix a frame $\{k,l,m_2,m_{3}\}$ as above on a neighborhood $U.$ By continuity we can restrict $U$ if necessary to a neighborhood $\tilde{U}$ such that $c$ is nowhere Kundt on $\tilde{U}.$ From this it follows that $\nabla_{a}k_{3}$ is nowhere zero on $c^{\perp}$ in the neighborhood. Therefore $S_{13}\equiv 0$ by eq. \eqref{S13Rel} and hence \begin{equation}\nabla_{1}S_{22}=\nabla_{2}S_{12}=\nabla_{3}S_{13}=\nabla_{2}S_{13}=\nabla_{0}S_{13}\equiv 0,\end{equation} by eqs.  \eqref{DSRel333}, \eqref{DDSRela}-\eqref{DDSRelb}, and from  \eqref{S11Rel} we obtain $S_{11}\equiv0$.  Consequently,
	\begin{equation}
	\nabla_{X}S_{11}=-2S_{01}\nabla_{X}l_{1}=0,
	\end{equation}
	for any vector field $X.$ Lastly, the Bianchi identity for conformally flat metrics,
	\begin{equation}\label{ConformalFlatBianchi}
	\nabla_{[a}S_{b]c}+\frac{1}{12}\nabla_{[a}Rg_{b]c}=0,
	\end{equation}
	shows that $0=2\nabla_{[1}S_{2]2}=-\frac{1}{12}\nabla_{1}R,$ and therefore using the Bianchi identity again to permute indices demonstrates that all boost-weight negative components of $\nabla S$ vanish. In conclusion $\{S,\nabla S\}$ is uniformly of type $D$ w.r.t. $c$ on $\tilde{U}$ which is a contradiction.  Hence $c$ is  Kundt everywhere on $M$.
	
	Next we prove $ii).$ 
	Given a point $p\in M$, let $\{k,l,m_2,m_3\}$ be a frame as constructed above in a neighborhood $U$ such that the negative boost-weight components of $S$ and $\nabla S$ vanish. Then
	\begin{align}
	0=&\nabla_{1}S_{23}=(-S_{33}+S_{22})\nabla_{1}(m_{2})_{3},\\
	0=&\nabla_{3}S_{12}=(S_{22}-S_{10})\nabla_{3}(m_2)_{1},\\
	0=&\nabla_{1}S_{12}=(S_{22}-S_{01})\nabla_{1}(m_{2})_{1},\\
	0=&\nabla_{2}S_{12}=(S_{22}-S_{01})\nabla_{2}(m_{2})_{1},
	\end{align}
	which together with \eqref{m2Rel} shows that
	\begin{equation}\label{m2id}
	\nabla m_{2}=\theta(kl+lk+m_{3}m_{3})=\theta h,
	\end{equation}
	where $h=g-m_2m_2$ and $\theta$
	is some smooth function  on $U.$ For later use let us show that \begin{equation}\label{Cid}\nabla_{a}\theta\propto (m_2)_a. 
	\end{equation}
	From \eqref{m2id} we have $\nabla_{d}(m_2)^{d}=3\theta$. The Ricci identity shows that
	\begin{equation}\label{CRicciId}
	\begin{gathered}
	R_{2d}=2\nabla_{[c}\nabla_{d]}(m_{2})^{c}=\nabla_{c}\nabla_{d}(m_2)^{c}-\nabla_{d}\nabla_{c}(m_{2})^{c}
	=\nabla_{c}\theta\tensor{h}{_{d}^{c}} -3\nabla_{d}\theta\\
	=(\nabla_{c}\theta)\tensor{h}{_d^{c}}-\theta\nabla_{c}[(m_{2})_{d}(m_{2})^{c}]-3\nabla_{d}\theta=(\nabla_{c}\theta)\tensor{h}{_d^{c}}-3\theta^2(m_{2})_d-3\nabla_{d}\theta,
	\end{gathered}
	\end{equation}
	which verifies \eqref{Cid} since $R_{2d}\propto (m_{2})_{d}.$
	
	Now consider the distribution on $U$ given by $\Delta_{q}=\Span\{k_q,l_{q},(m_{3})_q\}$, for each $q\in U.$ It is clear from \eqref{m2id} that $\Delta$ is integrable and $[m_2,\Delta]\subset \Delta$ implying that the flow of $m_2$ preserves integral surfaces of $\Delta.$ Furthermore the second fundamental form, $A,$ of $\Delta$ is given by
	\begin{equation}
	A(k,k)=0,\quad A(k,l)=-\theta m_2,\quad  A(k,m_3)=0,\end{equation}
	\begin{equation}
	A(l,l)=0,\quad A(l,m_3)=0,\quad A(m_3,m_3)=-\theta m_2.
	\end{equation}
	Let a hat mark the curvature tensors on integral surfaces of $\Delta,$ then using the Gauss-Codazzi equations,
	\begin{equation}
	g(R(X,Y)Z,W)=g(\hat{R}(X,Y)Z,W)+g(A(X,Z),A(Y,W))-g(A(Y,Z),A(X,W)),
	\end{equation}
	for all $X,Y,Z,W\in \Delta,$
	it is clear that the Ricci tensor, $\widehat{Rc}$,  is of type $D$ w.r.t. the frame $\{k,l,m_3\}.$ Furthermore,
	\begin{equation}
	\begin{gathered}
	\hat{R}_{01}=\hat{R}_{1001}+\hat{R}_{3031}=R_{1001}-g(A(l,k),A(k,l))+g(A(k,k),A(l,l))\\+R_{3031}-g(A(m_3,m_3),A(k,l))+g(A(k,m_3),A(m_3,l))={R}_{1001}+{R}_{3031}-2\theta^2,
	\end{gathered}
	\end{equation}
	\begin{equation}
	\begin{gathered}
	\hat{R}_{33}=2\hat{R}_{0313}=2[R_{0313}-g(A(k,l),A(m_3,m_3))+g(A(m_{3},l),A(k,m_3))]\\
	=2R_{0313}-2\theta^2.
	\end{gathered}
	\end{equation}
	The equation
	\begin{equation}\label{endeq}
	0=\langle Rm\vert k \vert klkm_3\rangle_{3}=-R_{3103}+R_{0101},
	\end{equation}
	shows that $\hat{R}_{01}=\hat{R}_{33}.$
	Hence, the integral surfaces of $\Delta$ are Einstein and therefore locally maximally symmetric.
	
	Now fix some integral surface $S$ passing through the point $p$ and let $\phi_t$ be the flow of $m_2.$ Restricting $S$ if necessary we can find an interval $I=(-\epsilon,\epsilon)$ such that $\phi_t$ is defined on $S$ for $t\in I.$ Now define $\Phi:I\times S\rightarrow M$ by $\Phi(t,q)=\phi_{t}(q),$ for all $t\in I$ and $q\in S.$
	Since $[m_2,\Delta]\subset \Delta$ the image $S_t:=\phi_t(S)$ is a new integral surface of $\Delta$ and by \eqref{m2id}, \eqref{Cid} the map $S\overset{\phi_t}{\rightarrow}S_t$ is a homothety. Consequently the pull-back of the metric to $I\times S$ by $\Phi$ gives the desired warped product.
	
\end{proof}

%{\bf Remark}. 
\begin{remark}\label{rem:4Dproof} \begin{em} An alternative, streamlined proof also utilizes %of theorem \ref{thm:tachyonic} can be streamlined using 
the Weyl-like {\em Pleb\'{a}nski tensor}
	$$
	P^{ab}{}_{cd}=S^{[a}{}_{[c}S^{b]}{}_{d]}+\delta^{[a}{}_{[c}S_{d]e}S^{b]e}-\tfrac16\delta^{[a}{}_{[c}\delta^{b]}{}_{d]}S^{ef}S_{ef},
	$$
where indices in square brackets are antisymmetrized. %This is a Weyl-like tensor. 
Suppose that condition a) of the theorem holds for some null congruence $c$, $\bo{S}=0$. Then either $P\neq0$ and $\bo{P}\leq 0$, or $P=0$.
%i.e., the null alignment type of the Ricci tensor is either II or D w.r.t.~a null congruence $c$, 
%There are three cases: %qualitatively different possibilities:
\begin{itemize}
	\item If $P\neq 0$ and $\bo{P}=0$ the Pleb\'{a}nski-Petrov type is either II or D. Proposition \ref{BilWeylOnto} implies $d_1^c=2$, such that condition c) is not met. However, note that if ${\nabla S}\leq 0$ then $\bo{\nabla P}\leq 0$, and theorem \ref{NonConformallyFlatKundtChar} applied to $W=P$ shows that $c$ is Kundt.
	\item If $P\neq 0$ and $\bo{P}=-1$ \{$\bo{P}=-2$\} the Pleb\'{a}nski-Petrov type is III \{N\}; if condition b) holds we have $\bo{\nabla^m}P\leq 0$ for $0\leq m\leq 3$, and we can apply proposition \ref{prop:Weyl-III} \{proposition \ref{prop:Weyl-N}\} to $W=P$ to conclude that $c$ is Kundt.
\end{itemize}
The only remaining case is $P=0$, which happens precisely when a unit spacelike vector field $u$ ($u^au_a=1$) exists such that 
	%the trace-free Ricci tensor has the structure
	$$
	S_{ab}=\lambda(u_au_b-\tfrac13 h_{ab}),\quad \lambda\neq 0,\quad h_{ab}=g_{ab}-u_au_b.
	$$
Note that $S^{ab}S_{ab}=\frac23\lambda^2$, such that $u_au_b=\frac14(\frac{3S_{ab}}{\lambda}+g_{ab})$ and
$h_{ab}=\frac34(g_{ab}-\frac{S_{ab}}{\lambda})$ belong to ${\cal C}_{2,1}$,
%the algebra spanned by $Rm$ and $g$, 
and that $\bo{S}=0$, 
$\langle uu|k|ku\rangle=-u$ and $d^{c}_1=1$ for any vector field $k$ orthogonal to $u$ and the null congruence $c$ it generates. %where $k$ generates $c$. 
The Bianchi identity \eqref{ConformalFlatBianchi} shows that 
%with $h_e{}^ah_f{}^bu^c$ and $h_e{}^au^bh_f{}^c$ show that 
the covariant derivative of $u$ has the structure $\nabla_b u_a=\dot{u}_a u_b+\theta h_{ab}$, where $\dot{u}^a\equiv u^b\nabla_b u^a$ is orthogonal to $u^a$, and $h_a{}^b\nabla_b\lambda=\lambda\dot{u}_a=\frac14 h_a{}^b\nabla_b R$ where $R$ is the Ricci scalar. Note that $\dot{u}^a=h^{bc}\nabla_c(u^au_b)$ belongs to ${\cal C}_{1,2}$
%the algebra spanned by $Rm$, $\nabla Rm$ and $g$, 
and is thus spacelike, null non-zero, or zero. 
\begin{enumerate}
	\item If $\dot{u}$ is spacelike then $\bo{\nabla S}=0$ holds precisely for %$c=c_1$ and $c=c_2$ 
	the two null directions $c=c_1,\,c_2$ of the timelike 2-plane orthogonal to $u$ and $\dot{u}$. For $k$ along one of these directions we have $\langle \dot{u}|k|k\rangle=-\dot{u}$, such that $d_2^c=2$ and condition c) cannot be met; however, $c$ is Kundt by corollary \ref{TopDim}.
	\item If $\dot{u}$ is null non-zero then $\nabla S$ is of genuine type II w.r.t.~$c$ generated by $\dot{u}$, and $d_2^c=1$. If condition b) holds then $c$ is Kundt by virtue of proposition \ref{prop:k-III}.
	\item If $\dot{u}=0$ then $\{S,\nabla S\}$ is uniformly of type D, \eqref{m2id} with $m_2=u$ holds, and continuing by \eqref{Cid}-\eqref{endeq} we find that the spacetime is locally a warped product with metric \eqref{warp}, where $uu=dx^2$ and $d_3^c=1$ for any $c$ orthogonal to $u$.
\end{enumerate} 
\end{em}
\end{remark}

We now combine previous results in the following theorem where there are no regularity assumptions on the algebraic types, i.e., algebraic types are free to change subject to the given conditions.
\begin{theorem}\label{thm:culmination4d}
	Let $(M,g,c)$ be a four dimensional Lorentzian manifold with a null-line distribution $c$ such that $\nabla^{m}Rm$ is algebraically special w.r.t. $c$, for $0\leq m\leq 3,$ and define $\Gamma $ to be the set of points $p\in M$ satisfying at least one of the following properties: 
	\begin{enumerate}[i)]
		\item the Weyl tensor does not vanish at $p$, $C_p\neq 0;$  
		\item $\{S,\nabla S\}$ is not uniformly of type $D$  w.r.t.\ $c$ at $p$;
		\item $d_{3}^{c}(p)=2.$
	\end{enumerate}
	Then $c$ has the Kundt property on the closure $\bar{\Gamma}$.
	
	Furthermore, for
	each point $p\in M/(\bar{\Gamma}\cup \{S=0\})$ there exists an open neighborhood $U$, an interval $I=(-\epsilon,\epsilon)$, a three-dimensional locally maximally symmetric Lorentzian manifold $(N,h)$ and a strictly positive smooth function $f:I\rightarrow \mathbb{R}$ such that $(U,g)$ is isometric to  $I\times N$ with the warped product metric
	\begin{equation}
	dx^{2} + fh,
	\end{equation}
	where $x$ is a coordinate on $I.$
\end{theorem}
\begin{proof}
	Let $p\in M$ be a point such that $c$ does not have the Kundt property at $p.$ By continuity we can find a neighborhood $U$ of $p$ such that $c$ is nowhere Kundt on $U$; by corollary \ref{KundtOnK} it follows that $d_3^c\leq 1$ on $U$, and by proposition \ref{BilWeylOnto} $C$ has boost order $\leq -1$ along $c$ on $U$. If there exists a point $p$ in $U$ such that $\bo{C_p}=-1$, then lower semi-continuity of the boost order shows that this holds on an open neighborhood $V\subset U$ of $p$, which gives a contradiction with proposition \ref{prop:Weyl-III}. The same reasoning shows that there cannot exist points in $U$ for which the boost order of $C$ is $-2$  and thus $C$ vanished identically on $U$. Similarly, the same inductive reasoning using theorem \ref{ConformallyFlatTachyonic} shows that $S$ is of type $D$ or zero at any point in $U$. Furthermore, if at any point $q\in U,$ $\{S_q,\nabla S_{q}\}$ is not uniformly of type $D,$ then this holds in a neighborhood about $q$ and by theorem \ref{ConformallyFlatTachyonic}, $c$ has the Kundt property at $q,$ which is a contradiction.
	
	It follows that $U\cap \Gamma =\emptyset,$ and therefore $p\notin \bar{\Gamma}.$ Hence, if $p\in \bar{\Gamma},$ then $c$ has the Kundt property at $p.$
	
	By definition, $p\in M/(\bar{\Gamma}\cup \{S=0\})$ iff.\ there exists a neighborhood $U$ of $p$ such that $\{S,\nabla S\}$ is  uniformly of type $D$, $C=0$ and $d_{3}^{c}=1$ everywhere on $U.$ By theorem \ref{ConformallyFlatTachyonic}, $U$ has the desired local warped product structure.
\end{proof}

\section*{Acknowledgments} The authors were supported by the Research Council of Norway, Toppforsk grant no.\ 250367:
Pseudo-Riemannian Geometry and Polynomial Curvature Invariants: Classification, Characterisation and Applications.

%\section{Discussion}\label{sec:discussion}
\medskip

\appendix 

\section{Boost order and null alignment classification of tensors}\label{app:nullalignment}

Let $(V,g)$ be a Lorentzian space, consisting of a real vector space $V$ with Lorentzian inner product $g$. For tensors over $V$ and the dual space $V^*$ we use index-free or abstract index notation, as deemed appropriate in the context; abstract indices are lowered and raised by the metric $g_{ab}$ resp.~the inverse metric $g^{ab}$ ($g_{ab}g^{bc}=\delta_a^c$), leading to geometrically equivalent tensors denoted by the same symbol. 

Consider a fixed null line (one-dimensional subspace) $c$ in $(V,g)$, and an arbitrary null vector $k\in c$. Take a null vector $l$ such that $g(k,l)=1$, let $\pi_{kl}$ denote the timelike 2-plane spanned by $k$ and $l$, and complete the basis $\{k,l\}$ of $\pi_{kl}$ to a (real) null frame $\{e_\alpha\}=\{e_0,e_1,e_i\}=\{k,l,m_i\}$ of $(V,g)$, where the spacelike vectors $m_i$ form an orthonormal basis of $\pi_{kl}^\perp$. Here and throughout the paper, (indexed) Greek frame labels run from 0 to $n-1$, while spatial frame labels $i,\,j,\,\ldots$ run from $2$ to $n-1$. 
Consider the endomorphism ${\cal X}$ of $V$ defined by $k\mapsto k,\,l\mapsto -l,\, m_i\mapsto 0$, which is a generator of the pure boosts in $\pi_{kl}$.
%  A (pure) boost in $\pi_{kl}$ is a Lorentz transformation $e^{\tau{\cal X}}$ on $(V,g)$, where $\tau\in\mathbb{R}$ and ${\cal X}$ is the infinitesimal boost generator $\{k,l,m_i\}\mapsto\{k,-l,0\}$. 
%The adjoint of ${\cal X}$ is the map ${\cal X}^*$ given by $(k,l,m_i)\mapsto(-k,l,0)$. 
Let ${\cal T}_r$ be the vector space of covariant tensors of rank $r$ over the dual space $V^*$. The tensor product of $r$ copies of the adjoint map of ${\cal X}$
%map ${\cal X}^*$ 
is a diagonalizable endomorphism of ${\cal T}_r$ with spectrum $\{b|-r\leq b\leq r\}$; the corresponding eigenspaces ${\cal T}_{r,b}$ depend on $\pi_{kl}$ and have a monomial basis $\{e_{\alpha_1}\cdots e_{\alpha_r}|\bw{\alpha_1\cdots\alpha_r}=-b\}$, where $\bw{\alpha_1\cdots\alpha_r}=\sum_{i=1}^r(\delta_{\alpha_i 0}-\delta_{\alpha_i 1})$ is the {\em boost weight} of the multi-index $\alpha_1\cdots\alpha_r$. By spectral decomposition any $T\in{\cal T}_r$ can be written uniquely as the sum 
\begin{equation}
T=\sum_{b=-r}^r (T)_b,\quad (T)_b\in{\cal T}_{r,b}.
\end{equation}      
Define the integer $s$ by $(T)_s\neq 0$ and $(T)_{b}=0$ for $b>s$. Since any other null vector $\hat{l}$ with $g(k,\hat{l})=1$ can be written as $\hat{l}=l-z^i m_i-\frac12 z^iz_i k$ it is easy to see that $s$ does not depend on the chosen $l$, nor does it depend on the choice of $k\in c$. Therefore this integer is an intrinsic property of $c$ and $T$, denoted $\bo{T}$ and called the {\em boost order} of $T$ along $c$. 

Let ${\cal N}$ denote the set of all null lines in $(V,g)$, and define 
\begin{equation}
\bomax{T}=\textrm{max}\{\bo{T}|c\in{\cal N}\}\,.
\,,\qquad \textrm{bo}_{\textrm{min}}(T)=\textrm{min}\{\bo{T}|c\in{\cal N}\}\,.
\end{equation}
We note that $\bomax{T}$ equals the number of totally antisymmetric index slots in $T_{a_1\cdots a_r}$. 
%If $\bo{T}<\bomax{T}$ then $c$ is called an {\em aligned null direction} of $c$. 
On the other hand, $\bomin{T}$ 
%Define
%\begin{equation}
%%\textrm{bo}_{\textrm{max}}(T)=\textrm{max}\{\bo{T}|c\in{\cal N}\}\,.
%%\,,\qquad 
%\textrm{bo}_{\textrm{min}}(T)=\textrm{min}\{\bo{T}|c\in{\cal N}\}\,.
%\end{equation}
%This integer 
can take any value between $-\bomax{T}$ and $\bomax{T}$,
%$-\textrm{bo}_{\textrm{max}}(T)\leq \textrm{bo}_{\textrm{min}}(T)\leq \textrm{bo}_{\textrm{max}}(T)$ 
depending on the specific tensor $T$. 
If $\bomin{T}\leq 0$ then $T$ is called of (null alignment) {\em type II or more special}. Consider the subcase where $\bomin{T}=0$; if there is a unique $c\in{\cal N}$ such that 
%$\bo{T}=0$ 
$\bo{T}=\bomin{T}=0$ then $T$ is of {\em genuine type II}, else it is of {\em type D}. $T$ is of {\em genuine type III} if $-\bomax{T}<\bomin{T}<0$ and of {\em genuine type N} if $\bomin{T}=-\bomax{T}$; in these cases the null line $c$ for which $\bo{T}=\bomin{T}$ is automatically unique.

Consider a collection of tensors $\{T[A]|A\in{\cal A}\}$. If there exists $c\in{\cal N}$ such that $\bo{T[A]}\leq 0$ for all $A\in {\cal A}$ then the tensors are of {\em aligned type II or more special} (w.r.t.~$c$). If there exist different null lines $c_1,c_2\in{\cal N}$ such that $\textrm{bo}_{c_1}(T[A])=\textrm{bo}_{c_2}(T[A])=0$ for all $A\in {\cal A}$ then the tensors are said to be of {\em aligned type D} (w.r.t.~$c_1$ and $c_2$).

%\bibliography{KundtTheorem34D}{}
%\bibliographystyle{plain}

\end{document}